\documentclass[12pt]{amsart}
\usepackage[textwidth=6.375in,margin=1in]{geometry}
\usepackage{amssymb}
\usepackage{amsthm}
\usepackage{amsmath}
\usepackage{amscd}
\usepackage{verbatim}
\usepackage[all]{xy}
\usepackage{longtable, lscape}
\usepackage{todonotes}

\usepackage{multirow,array}
\usepackage[urlcolor=blue]{hyperref}

\numberwithin{equation}{section}
\numberwithin{table}{section}

\theoremstyle{plain}
\newtheorem{theorem}{Theorem}[section]
\newtheorem{corollary}[theorem]{Corollary}
\newtheorem{lemma}[theorem]{Lemma}
\newtheorem{proposition}[theorem]{Proposition}

\theoremstyle{definition}

\newtheorem{example}[theorem]{Example}

\theoremstyle{remark}
\newtheorem{remark}[theorem]{Remark}

\newcommand{\A}{\mathbb{A}}
\newcommand{\R}{\mathbb{R}}
\newcommand{\Q}{\mathbb{Q}}
\newcommand{\Z}{\mathbb{Z}}
\newcommand{\N}{\mathbb{N}}
\newcommand{\C}{\mathbb{C}}

\renewcommand{\H}{\mathbb{H}}

\newcommand{\D}{\mathbb{D}}

\newcommand{\zxz}[4]{\begin{pmatrix} #1 & #2 \\ #3 & #4 \end{pmatrix}}

\newcommand{\leg}[2]{\left( \frac{#1}{#2} \right)}
\newcommand{\kzxz}[4]{\left(\begin{smallmatrix} #1 & #2 \\ #3 & #4\end{smallmatrix}\right) }

\newcommand{\mat}[1]{\begin{pmatrix} #1 \end{pmatrix}}

\newcommand{\calA}{\mathcal{A}}

\newcommand{\calE}{\mathcal{E}}

\newcommand{\calN}{\mathcal{N}}
\newcommand{\calO}{\mathcal{O}}
\newcommand{\calP}{\mathcal{P}}

\newcommand{\calZ}{\mathcal{Z}}

\newcommand{\fraka}{\mathfrak a}

\newcommand{\frake}{\mathfrak e}

\newcommand{\frakp}{\mathfrak p}

\newcommand{\eps}{\varepsilon}
\newcommand{\bs}{\backslash}
\newcommand{\norm}{\operatorname{N}}

\newcommand{\tr}{\operatorname{tr}}
\newcommand{\id}{\operatorname{id}}

\newcommand{\Cl}{\operatorname{Cl}}

\newcommand{\SL}{\operatorname{SL}}
\newcommand{\Sp}{\operatorname{Sp}}

\newcommand{\GSpin}{\operatorname{GSpin}}
\newcommand{\Mp}{\operatorname{Mp}}
\newcommand{\Orth}{\operatorname{O}}

\newcommand{\Hom}{\operatorname{Hom}}
\newcommand{\Aut}{\operatorname{Aut}}
\newcommand{\Mat}{\operatorname{Mat}}
\newcommand{\Spec}{\operatorname{Spec}}

\newcommand{\End}{\operatorname{End}}

\newcommand{\dv}{\operatorname{div}}

\newcommand{\Sym}{\operatorname{Sym}}

\newcommand{\Pet}{\text{\rm Pet}}

\newcommand{\GL}{\operatorname{GL}}
\newcommand{\SO}{\operatorname{SO}}

\newcommand{\Gal}{\operatorname{Gal}}

\newcommand{\ord}{\operatorname{ord}}

\newcommand{\diag}{\operatorname{diag}}
\newcommand{\co}{\mathcal  O}
\newcommand{\cha}{\operatorname{Char}}

\newcommand{\GSp}{\operatorname{GSp}}

\font\cute=cmitt10 at 12pt 
\newcommand{\kay}{{\hbox{\cute k}}}

\newcommand{\Fal}{\operatorname{Fal}}

\begin{document}

\title[Bad Reduction of Genus Two Curves with complex multiplication]{Bad Reduction of Genus Two Curves with complex multiplication}

\author[Jan H.~Bruinier, Tonghai Yang, and Peng Yu]{Jan
Hendrik Bruinier, Tonghai Yang, and Peng Yu}
\address{Fachbereich Mathematik,
Technische Universit\"at Darmstadt, Schlossgartenstrasse 7, D--64289
Darmstadt, Germany}
\email{bruinier@mathematik.tu-darmstadt.de}
\address{Department of Mathematics, University of Wisconsin Madison, Van Vleck Hall, Madison, WI 53706, USA}
\email{thyang@math.wisc.edu}
\address{School of Mathematics, Renmin University of China, No. 59, Zhongguancun Street, Haidian District, Beijing, China}
\email{yupeng2020@ruc.edu.cn}

\thanks{The first author is supported in part by  the
DFG Collaborative Research Centre TRR 326 ``Geometry and Arithmetic of Uniformized Structures'', project number 444845124. The second author is partially supported by UW-Madison
Kellett Mid-Career Award and an NSF grant  DMS-2501617.}



\begin{abstract}
We study genus two curves whose Jacobians have complex multiplication by a biquadratic CM field.
In this setting the Jacobian decomposes as a product of two elliptic curves with complex multiplication by orders in the same imaginary quadratic field.
We derive upper bounds for the primes of stable bad reduction of such curves in terms of the discriminants of the CM orders.
To do so, we give an explicit description of principally polarized abelian surfaces with CM by a biquadratic field via small CM cycles on an orthogonal Shimura variety of signature $(3, 2)$. We use it to prove an arithmetic intersection formula between such CM cycles with Humbert surfaces in terms of coefficients of incoherent Eisenstein series and representation numbers of ternary positive definite quadratic forms.
Finally, employing the arithmetic properties of higher Green functions, we show that primes of stable bad reduction always exist.
\end{abstract}

\maketitle

\section{Introduction}

By work of Serre and Tate \cite{ST68}, any abelian variety with complex multiplication over a number field  has potentially good reduction everywhere.
If $C$ is a (smooth, geometrically connected, projective) curve over a field $k$, then its Jacobian $J(C)$, together with its theta divisor $\theta_C$, is a principally polarized abelian variety over $k$. We say that $C$ has complex multiplication  (short `CM') if the base change of $J(C)$ to an algebraic closure of $k$ has CM. Thus, if $C$ is a curve over a number field with CM, then $J(C)$ has potentially good reduction everywhere.
However, if the genus of $C$ is greater than one, this does not imply that $C$ itself has potentially good reduction everywhere. There may instead be
primes of stable bad reduction, that is, primes at which the stable model of $C$ has non-smooth special fiber.
%

The main purpose of the present paper is to study the reduction of genus two curves  with CM by a biquadratic field systematically. This is motivated by recent work of Lorenzo Garc\'ia, Ritzenthaler, and Rodr\'iguez Villegas \cite{GRV}, who consider the special case in which  $J(C)\cong E\times E$ where $E$ is an elliptic curve with CM by the maximal order of an imaginary quadratic field of discriminant $d<0$. They prove that the primes $p$ of stable bad reduction are bounded by $p\leq |d|/4$.   This result may be viewed as an analogue of the celebrated theorem of Gross and Zagier on prime factorizations of singular moduli \cite{GZ}.

In this paper we treat the general case in which $C$ has complex multiplication by a biquadratic CM field. Then $J(C) \cong E_1 \times E_2$ is the product of two CM elliptic curves  with CM by the same  imaginary quadratic field
(but with possibly different CM orders), and the polarization $\theta_C$ is {\em not}\/ equivalent to the product polarization, see e.g.~\cite[Corollary 10.6.3]{BL}. For a negative discriminant $\delta\in \Z$, we write $\calO_\delta$ for the quadratic order of discriminant $\delta$ in $\Q(\sqrt{\delta})$. Our first main result is as follows.

\begin{theorem} \label{thm:BadReduction}
Let $C$ be a genus two curve
over a number field such that $J(C) \cong E_1 \times E_2$, where $E_1$ and $E_2$ are CM elliptic curves with $\End (E_i) =\co_{d_i}$ and $\Q(\sqrt{d_1}) =\Q(\sqrt{d_2})$. Let $\operatorname{lcm}(|d_1|, |d_2|)$ be the least common multiple of $|d_1|$ and $|d_2|$. If $C$ has stable bad reduction at a rational prime  $p$,  then $p\le \frac{\operatorname{lcm}(|d_1|, |d_2|)}4$. Moreover, $p|\frac{\operatorname{lcm}(|d_1|, |d_2|)}4$ or  $p$ is inert in  $\Q(\sqrt{d_1})$.
\end{theorem}

A slightly stronger version of this theorem is given in Theorem  \ref{theo:Curve}(2) and Remark \ref{rem7.2}. This theorem also proves and refines \cite[Conjecture 1.2]{KLLNOS} in the case where the genus is two and the CM type is non-primitive (with  $e=1/2$ in their notation).

On the other hand, our second main result shows that  there always exist primes of stable bad reduction.

\begin{theorem} \label{theo:AlwaysBad}
 Let $C$ be a genus two curve over a number field $K$ such that  its Jacobian $J(C)$ has CM by  a biquadratic CM field. Then  $C$ has stable bad reduction at some prime $p$.
\end{theorem}

This result is somewhat surprising when compared with the experimental data in the case of genus two curves with CM by a {\em primitive} quartic CM field. In this setting, examples of genus two curves with potentially good reduction everywhere are known, such as the curve $y^2 = x^5-1$, which has complex multiplication  by the cyclotomic field $\Q(\zeta_5)$. For further examples, also for non-Galois quartic CM fields, we refer to \cite{Wa99} and \cite{BS15}.
%
In the case of cyclic CM fields, Habegger and Pazuki proved that such examples are rare, see \cite[Theorem~1.1]{HP}. More specifically, they showed that for every fixed real quadratic field there are only finitely many genus two curves over $\bar \Q$  with potentially good reduction everywhere which have CM by the maximal order of a cyclic quartic CM field containing the given real quadratic field. It is expected that an analogous result also holds for more general orders and for non-Galois CM fields.

\medskip

To prove Theorems~\ref{thm:BadReduction} and~\ref{theo:AlwaysBad} above,
we first reformulate the question of stable bad reduction as an arithmetic intersection problem.
It is well-known that a principally polarized abelian surface is not isomorphic to the Jacobian of a genus two curve $C$ if and only if it is decomposable, that is, isomorphic to the product of two elliptic curves equipped with the product polarization.
Hence, a genus two curve $C$ over a number field with CM has stable bad reduction at a prime,
exactly if
its  Jacobian $(J(C), \theta_C)$ becomes isomorphic modulo
that prime to such a product.
Equivalently, on the moduli stack   of principally polarized abelian surfaces, the CM cycle  determined by the N\'eron model of $J(C)$ and the divisor given by the decomposable locus
have non-zero intersection
at this prime.

We show that this question can also be described  as an arithmetic intersection problem on an orthogonal Shimura variety of signature $(3,2)$ between a small CM cycle (as in \cite{Schofer}, \cite{BY09}) determined by $J(C)$  and the special divisor corresponding to the Humbert surface of discriminant $1$. More generally, we shall study intersections with Humbert surfaces of arbitrary discriminant. Using a computation of the Faltings height (see \cite{Mocz}, \cite{NT}, and Section 5.1 for an alternative approach) together with Schofer's formula for evaluating  Borcherds products at small CM cycles, we determine the required arithmetic intersection in this general setting. Explicit formulas for coefficients of derivatives of incoherent Eisenstein series then yield the bound on the primes of bad reduction stated in Theorem~\ref{thm:BadReduction}. Finally, we employ the CM value formula for higher Green functions from \cite{BEY} to prove Theorem~\ref{theo:AlwaysBad}.

We now describe our approach and the results in more detail.
Let $\mathbf A=(A, \lambda)$ be a principally polarized abelian surface over $\C$.
The principal polarization $\lambda$ determines  a Riemann form on
$\Lambda =H_1(A, \Z)$, which gives rise to the Rosati involution  $f \mapsto f^\dagger$ on $\End(\Lambda)$. It can be proved (see Section \ref{sect:set-up} for details) that
$$
L=L(\mathbf A) =\{ f \in  \End(\Lambda):\;  \text{$f^\dagger =f$ and $\tr(f) =0$}  \}
$$
equipped with the quadratic form  $Q(f)=\frac{1}4 \tr(f^2)$ is an even lattice of signature $(3, 2)$. Its discriminant group
 $L'/L$ is isomorphic to $\tfrac{1}{2}\Z/\Z$.
The sublattice  $\mathcal P =\mathcal P(\mathbf A) =L\cap \End(A)$ of endomorphisms compatible with the complex structure on $A$ is positive definite. Following  Kudla and Rapoport \cite{KRSiegel} we call it the lattice of \emph{special endomorphisms} of $\mathbf A$. We also let $\mathcal N$ be the orthogonal complement of $\mathcal P$ in $L$.

If $\mathbf{A}$ has  CM by a biquadratic field, then, as mentioned above, there exists a unique imaginary quadratic field $\kay$ of discriminant $d$ such that $A\cong E_1\times E_2$ is the product of two elliptic curves with CM by (possibly different)  orders
in the same imaginary quadratic field $\kay$. More specifically,  there exists a rank $2$ integral lattice  $\mathfrak a\subset \kay$ with endomorphism ring $ \co_{d_2}\subset \calO_d$ and a positive integer $t$, such that $A$ is isomorphic to
\[
A(\mathfrak a, t) :=\C/\co_{t^2d_2} \times \C/\mathfrak a
\]
over $\C$.  We abbreviate the discriminant of the first order by $d_1=t^2 d_2$. The pair $(\mathfrak a, t)$ characterizes $A$ up to isomorphism (see Proposition \ref{prop:product}).
In this case, it can be proved that
$\mathcal P $ is a positive definite sublattice of $L$ of rank $3$, and its orthogonal complement $\mathcal N$ is a binary lattice  isomorphic to $(\bar{\mathfrak a}, -\frac{t\norm(\cdot)}{\norm(\mathfrak a)})$  (see  Proposition \ref{prop:N}). It turns out that $\calN$ is independent of the polarization on $A$, while $\calP$ depends on it.

There is  a natural isomorphism $ \GSp_{4}(\Z)\cong \GSpin(L)$, which induces a canonical isomorphism between  the Siegel threefold $\mathcal A_{2,\Q}$ and the GSpin-Shimura variety $X_L$  associated to $L$ over $\Q$. Here $\calA_g$ denotes the moduli stack over $\Z$ of principally polarized abelian schemes of dimension $g$.
 In Section \ref{sect:smallCM} we prove the following result
which may be of independent interest.


\begin{theorem} \label{theo:CMIdentification} (See also Theorem \ref{eq:GalAction}.)
Let $\mathbf A=(A(\mathfrak a, t), \lambda) \in \mathcal A_2(\C)$ be a principally polarized abelian surface as above.
\begin{itemize}
\item[(1)] Then $\mathbf A$ is defined over the ring class field $H_{d_1}$ of $\co_{d_1}$.
\item[(2)]  Under the identification $\mathcal A_{2,\Q}\cong X_L$ the cycle $Z(\mathbf A)$ of all Galois conjugates of $\mathbf A$ can be identified with the small CM cycle $Z(\mathcal N)$ associated to $\mathcal N$ as in Section~\ref{sect:OsmallCM}.
\end{itemize}
\end{theorem}

Let $K$ be a finite field extension of $H_{d_1}$ such that the N\'eron model $\underline{\mathbf A}$ of $\mathbf A$ over $\co_K$ is smooth, i.e., has good reduction everywhere, and defines a point  $\underline{\mathbf A} \in  \mathcal A_2(\co_K)$. We shall study the arithmetic intersection of this arithmetic curve with the divisor  given by the reducible locus, and more generally, with Humbert surfaces in $\calA_2$.
For a positive integer $m$, the Humbert surface  $\mathcal Z(m)$ of discriminant $m$ is defined as the moduli stack of pairs $(\mathbf B, f)$ where $ \mathbf B$ is a principally polarized abelian surface and $f\in \calP(\mathbf{B})$ is a special endomorphism
of $\mathbf B$ with $Q(f) =m$.
The  forgetful map from $\mathcal Z(m)$ to $\mathcal A_2$ is finite and unramified and defines a special divisor on $\mathcal A_2$ (see Section \ref{sect:Borcherds}). For example, $\mathcal Z(1)\cong \calA_1\times \calA_1$ is the reducible locus.
We are interested in the normalized arithmetic intersection number
\begin{equation}
\frac{1}{[K:H_{d_1}]}\underline{\mathbf A} \cdot \mathcal Z(m) = \frac{1}{[K:H_{d_1}]} \sum_{\mathfrak p }  e_{\mathfrak p} \log \norm(\mathfrak p)
\end{equation}
when  the intersection is proper, that is,  $\mathbf A \notin \mathcal  Z(m)(\C)$.  Here the sum runs over all prime ideals $\frakp$ of $\co_K$, and $ e_{\mathfrak p}$ is the largest non-negative integer $n$ such that $\underline{\mathbf A}\in \mathcal Z(m) $ modulo $ \mathfrak p^n$. The intersection number does not depend on the choice of $K$. Our third main result is the following intersection  formula, which we prove in Section~\ref{sect:MainFormula} (see Theorem~\ref{theo: mainIntSect}).

\begin{theorem}
\label{theo: MainFormula}
Let the notation be as above, and assume that $\mathbf A \notin \mathcal  Z(m)(\C)$.
Then
\begin{align}
\label{eq:intro0}
\frac{\underline{\mathbf A} \cdot \mathcal Z(m)}{[K:H_{d_1}]}
&=-\frac{h_{d_1} }{2 }
\sum_{\substack{\mu_1 \in \mathcal N'/\mathcal N\\ \mu_2 \in \mathcal P'/\mathcal P \\
\mu_1+\mu_2\equiv \frake_{m}\,(L)}}\sum_{\substack{m_1,m_2\in \Q_{\ge 0}\\ m_1 + m_2 = \frac{m}{4}\\  (m_1,\mu_1) \ne (0,0)}}
	\kappa_{\mathcal N}(m_1, \mu_1) a_\mathcal P(m_2, \mu_2).
%
\end{align}
Here $h_{d_1}$  denotes the ring class number  of $\co_{d_1}$, and $\frake_m\in L'/L$ the neutral (respectively non-neutral) element if $m$ is even (respectively odd).  Moreover, $a_{\mathcal P}(m_2, \mu_2)$ is
the representation number of $m_2$ by the coset $\mu_2+\calP$ of the positive definite lattice $\calP$, and  $\kappa_{\mathcal N}(m_1, \mu_1)$ is the $(m_1, \mu_1)$-th Fourier coefficient of the central derivative of the incoherent Eisenstein series associated to $\mathcal N$, defined in Section \ref{sect:inceis}.

The cycles $\underline{\mathbf A}$ and $\mathcal Z(m)$ do not intersect in the fiber above $p$ unless $p|d_1$ or $p \le \frac{m|d_2t|}4$ is inert in $\Q(\sqrt{d_1})$.
\end{theorem}

For $m=1$, we actually obtain a slightly better bound on the primes $p$,
see Theorem~\ref{theo:Curve}, which implies Theorem~\ref{thm:BadReduction} immediately.
If in addition the discriminant of $\kay$ is equal to $d=-\ell$  for a prime $\ell \equiv 3\pmod{4}$, we derive  a much simpler formula, which was conjectured by  Lorenzo Garc\'ia, Ritzenthaler, and Rodríguez Villegas \cite[Conjectural Formula (7.1)]{GRV}. This formula does not hold for general discriminants as they  observed.

\begin{corollary}
\label{prop:Conjecture}
Let  $\ell \equiv 3\pmod{4}$ be a prime.
Let $C$ be a genus two curve over a number field such that $J(C) \cong E_1 \times E_2$ with $\End(E_1) =\End(E_2) =\mathcal O_{-\ell}$.  Let $K$ be a number field over which $J(C)$ has good reduction everywhere. Then the normalized arithmetic intersection is given by
$$
\frac{\underline{\mathbf J(C)} \cdot (\mathcal A_1 \times \mathcal A_1)}{[K:H_d]}
=-\sum_{ \substack{x\in\mathcal P'\\ \frac{\ell -4\ell Q(x)}{4}\in\Z_{>0}}}
 \sum_{n|\frac{\ell-4 \ell Q(x)}4}\left(\frac{-\ell}{n}\right)\log n.
$$
Here $Q$ is the positive definite quadratic form on $\mathcal P'$ defined in (\ref{eq:L}).
\end{corollary}

Note that similar intersection problems were studied earlier in different settings for {\em primitive} quartic CM fields. For instance, Goren and Lauter studied class invariants of such CM fields that are constructed as CM values of certain Siegel modular functions. They provided upper bounds on the primes appearing in their denominators and used these bounds to construct $S$-units in primitive quartic CM fields, see~\cite{GL06}, \cite{GL07}, \cite{GL12}. The second author computed such intersections to prove the Colmez conjecture for primitive quartic CM fields (under a technical assumption on the discriminant), see~\cite{Ya10}, \cite{YaCol}. We intend to use our new approach to study this non-biquadratic case in a subsequent paper.


\medskip

The basic idea of the proof of Theorem \ref{theo: MainFormula} is as follows.
First, there is a unique weakly holomorphic modular form $f_m(\tau) = q^{-m} + c(0) + O(q)$ of weight $-1/2$ (for $\Gamma_0(4)$ in the Kohnen plus-space), whose Borcherds lifting  $\Psi_m(z)$  determines a Siegel modular form with divisor $\mathcal Z(m)$ over $\Spec(\Z)$. Using the properties of the Faltings height on $\mathcal A_2$, the arithmetic intersection number can be expressed as
\begin{align}
\label{eq:intro1}
\frac{\underline{\mathbf A} \cdot \mathcal Z(m)}{[K:H_{d_1}]} = \frac{c(0)  h_{d_1}}2 h_{\Fal}(A) +  \sum_{z \in Z(\mathbf A) } \| \log \Psi_m(z)\|_{\Pet}.
\end{align}

Since the Faltings height $h_{\Fal}(A)$ of the CM abelian surface $A$ is independent of the choice of the principal polarization, it can also be computed by viewing it as a point on the decomposable locus  $\mathcal A_1 \times \mathcal A_1$, which we regard as a GSpin-Shimura variety associated to the even unimodular lattice of signature $(2,2)$. We show in Proposition \ref{prop5.3} that
\begin{align}
\label{eq:intro11}
h_{\Fal}(A) =  \frac{1}2 \kappa_{\mathcal N}(0, 0) + \frac{1}2 (\log 4\pi + \Gamma'(1)),
\end{align}
where $\kappa_{\mathcal N}(0, 0)$ is given by the constant term of the derivative of the incoherent Eisenstein series associated with $\calN$.

Finally,  we identify  $Z(\mathbf A)$ with the small CM cycle $Z(\mathcal N)$ by means of Theorem~\ref{theo:CMIdentification}, and employ Schofer's CM value formula (see Theorem  \ref{thm:SmallCM}) to relate the evaluation of  $\| \log \Psi_m(z)\|_{\Pet}$ on the cycle $Z(\mathbf A)$ with $\kappa_{\mathcal N}(0, 0)$ and the sum on the right hand side of \eqref{eq:intro0}.
Putting these identities together, we derive  the theorem.

\begin{remark}
An alternative approach to prove Theorem \ref{theo: MainFormula},
more along the argument of \cite{GRV}, might be to use an identification of $\calA_2$ with the {\em integral model} of the orthogonal Shimura variety over $\Z$ together with an integral version of Theorem~\ref{theo:CMIdentification}, and to apply the intersection formula of \cite[Theorem 4.5.1]{AGHM}. This would require a generalization of the latter result to include non-maximal CM orders  and possibly even discriminants.
%
\end{remark}

We use Theorem \ref{theo: MainFormula} to show the positivity  of the normalized arithmetic intersection number. To this end, we notice that by a result of Y.~Li \cite{Li21}, the products $\kappa_\mathcal N(m_1,  \mu_1) a_{\mathcal P}(m_2, \mu_2)$ appearing in \eqref{eq:intro0}
are always non-positive. However, it is not clear at all that there is at least one non-vanishing summand. We shall prove this by adapting an idea of \cite{Li21} to our situation. We show that the evaluation of a certain higher Green function for the Humbert surface $\calZ(m)(\C)$ at the small CM cycle $Z(\calN)$ is always positive, see Corollary \ref{cor:phi1pos}. By means of the CM value formula for higher Green functions of \cite{BEY}, it can be finally deduced that there exists a non-vanishing summand.

%

\begin{theorem}
\label{theo:positive}
Let the notation and assumption be as in  Theorem \ref{theo: MainFormula}. Then
$$
\frac{\underline{\mathbf A} \cdot \mathcal Z(m)}{[K:H_{d_1}]} >0.
$$
\end{theorem}

Note  that Theorem~\ref{theo:AlwaysBad} above can be deduced from Theorem~\ref{theo:positive} by specializing to the case $m=1$.

\begin{example}
The curve of genus two defined by the equation
\[
C:\; y^2= x^6+x^3+\frac{4}{17}
\]
has CM by a biquadratic field containing $\kay=\Q(\sqrt{-51})$. The class number of $\kay$ is two, and it turns out that the Jacobian of $C$ is isomorphic to $\C/\calO_{-51}\times \C/\mathfrak{a}$, where $\fraka\subset  \calO_{-51}$ is a representative for the nontrivial element of the class group of $\kay$. A computation with Sage or Magma shows that a minimal equation for $C$ is
$y^2 = 289 x^6 + 289 x^3 + 68$.
It has minimal discriminant $2^{12} \cdot 3^6 \cdot  17^{15}$ and
conductor $2^4 \cdot 3^6 \cdot 17^4$. The normalized arithmetic intersection number  between $\underline{J(C)}$ and $ \mathcal Z(1)$ as in Theorem \ref{theo: MainFormula} is given by $\log(2^4\cdot 3^2)$. This implies that $C$ has stable bad reduction at $2$ and $3$.  Although $C$ has bad reduction at $17$, it has potentially good reduction there. Indeed, it has good reduction at all primes of $\Q(\sqrt{-51},  17^{\frac{1}6})$ above~$17$.  Notice that  the primes $2$ and $3$ satisfy the conditions in Theorem \ref{thm:BadReduction}.
For further examples we refer to Section~\ref{sect:example}.
\end{example}

The present paper is organized as follows. In Section \ref{sect:2}, we provide some background on Siegel threefolds, orthogonal Shimura varieties, and small CM cycles. In particular, we prove 
that the  cycle $Z(\mathbf A)$ of all Galois conjugates of $\mathbf A$ can be identified with the small CM cycle $Z(\mathcal N)$ on the orthogonal Shimura variety. In Section \ref{sect:lattice}, we provide explicit descriptions of the definite lattices $\calN$ and $\calP$ in terms of the CM abelian variety $\mathbf A$.
This is a rather  subtle part of our paper, making the main intersection formula effectively computable. In Section~\ref{sect:Review},
we review some important properties of the Borcherds lift,  incoherent Eisenstein series,
and Schofer's CM value formula in our setting.
These are necessary preparations for Section~\ref{sect:MainFormula}, where we prove the main arithmetic intersection formula Theorem~\ref{theo: MainFormula}. In Section~\ref{sect:positive}, we recall some facts on higher Green functions in our setting and employ the formula for their CM values of \cite{BEY}
to prove Theorem~\ref{theo:positive} following an idea of \cite{Li21}. In Section~\ref{sect:BadReduction}, we recall some facts on the reduction of algebraic curves and prove Theorem~\ref{thm:BadReduction}, Theorem~\ref{theo:AlwaysBad}, and Corollary~\ref{prop:Conjecture}. Finally, in  Section~\ref{sect:example} we provide some numerical examples and observations.

\subsubsection*{Acknowledgements}
We thank  Elisa  Lorenzo Garc\'ia, Eyal Goren, Ben Howard, Harun Kir,  Yingkun Li, Christophe Ritzenthaler, Fernando Rodríguez Villegas, Ananth Shankar, 	Ari Shnidman,   Dongxi Ye, and Wei~Zhang for their help.


\section{ Siegel threefolds,  orthogonal Shimura varieties, and small CM cycles}
\label{sect:2}

Here we provide some background on orthogonal Shimura varieties and Siegel threefolds. In particular, we study abelian surfaces with complex multiplication by a biquadratic field and relate the cycles of their Galois conjugates to small CM cycles on an orthogonal Shimura variety.

\subsection{Orthogonal Shimura varieties and small CM cycles}
\label{sect:OsmallCM}
Let $L$ be an even quadratic lattice of signature $(n, 2)$ and $V=L \otimes_\Z \Q$ be the associated quadratic space over $\Q$. Let $\mathbb D$ be the Grassmannian of oriented negative $2$-planes in $V_\R$, which we identify with its projective model  $\mathcal L/\C^\times$, where
$$
\mathcal L = \{ w \in V_\C:\,  (w, w) =0,  (w, \bar w)<0\}.
$$
Given $w = u + i v \in \mathcal L$, one has $(u, u)=(v, v) <0$ and $(u, v) =0$. Then  the associated oriented negative $2$-plane in $\mathbb D$ is given by $\R u + \R v$. The space $\mathbb D$ has two connected components. We fix one of them and denote it by $\mathbb D^+$.
Let $G=\GSpin(V)$, and let $K=G(\hat\Z)  \subset G(\A_f)$ be the compact open subgroup given by those elements which preserve  $ \hat L $ and act trivially on $L'/L$. Let $X=X_L $ be the associated Shimura variety over $\Q$ with
\[
X_L(\C) = G(\Q) \backslash \mathbb D \times G(\A_f)/K.
\]

If $\mathcal N \subset L$ is a sublattice of signature $(0, 2)$,
then  $\mathcal N_\R$ is a negative $2$-plane in $V_\R$ and, when equipped with an orientation, thus gives rise to two `small' CM points $z_\mathcal N^{\pm}$.  Notice that $T =\GSpin(\mathcal N_\Q) \cong \kay^\times$,  where $\kay=\Q(\sqrt{-\det \mathcal N})$ is the even Clifford algebra of $\mathcal N_\Q$.  Moreover,
$$
K_T= K\cap T(\A_f)=\{ a \in \hat{\co}_{\kay}^\times :\,  (a/\bar a) \hat L =\hat L\}
$$
is a compact open subgroup of $T(\A_f)$.  Let
\begin{align}
\label{eq:smallcm}
Z(\mathcal N)  = \{ z_\mathcal N^{\pm} \}\times  T(\Q) \backslash T(\A_f)/K_T
\end{align}
be the small CM cycle  in $X_L$ associated with $\mathcal N$ as in \cite{Schofer}, \cite{BY09}. It is defined over $\Q$. Notice that $\Cl(\mathcal N)= T(\Q) \backslash T(\A_f)/K_T$ is the class group   associated  to $K_T$.
The group $\Cl(\mathcal N)$ acts on the $Z(\mathcal N)$ via multiplication on itself and leaves $z_{\mathcal N}^\pm$ fixed. Write $Z(\mathcal N) = Z^+(\mathcal N) + Z^-(\mathcal N)$ with $Z^+(m, \mu) =\{ z_{\mathcal N}^+\} \times  T(\Q) \backslash T(\A_f)/K_T$ defined over $\kay$, and similarly for $Z^-(\mathcal N)$.  The complex conjugation on $\kay$ maps the two cycles  $Z^\pm(\mathcal N)$ to each other. A point in  $Z(\mathcal N)$ is called a \emph{small CM point} (with respect to $\kay$).

Recall that
for $m \in \Q_{>0}$ and $\mu \in L'/L$ with $Q(\mu) \equiv m \pmod 1$, there is a special divisor $Z(m, \mu)$ on $X_L$. If  $X_L$ has only one connected component $X_L(\C) = \Gamma \backslash \mathbb D^+$ with $\Gamma = K \cap H(\Q)^+$ (which is what we need here), it can be defined as
$$
Z(m, \mu) =\Gamma \backslash \{ z \in  \mathbb D^+:\, \text{$(z, x) =0$ for some $ x \in  \mu+ L$  with $Q(x) =m$}\}.
$$
From now on, we assume that $X_L(\C) = \Gamma \backslash \mathbb D^+$. The following lemma  is  clear,  and  will be useful later in this paper.

\begin{lemma}
\label{lem:inclusion}
Let the notation be as above, and let $\mathcal P = \calN^\perp \cap L$ be the orthogonal  complement  of $\mathcal N$ in $L$, which is a positive definite even lattice of rank $n$. Then  $Z(\mathcal N) \cap Z(m, \mu)$ is non-empty if and only if $\mu \in \mathcal P'/\mathcal  P$
and   the representation number $a_{\mathcal P}(m, \mu)$ of $m$ by the coset $\mu +\mathcal P$ is non-zero.  In such a case, $Z(\mathcal N) \subset Z(m, \mu)$.
Here the condition $\mu \in  \mathcal P'/\mathcal P$ means that there exists a preimage  of $\mu$ in $L' \cap \mathcal P' $.
\end{lemma}

\subsection{The Siegel threefold as an orthogonal Shimura variety}
\label{sect:siegel3fold}
Let $\H_2^\pm$ denote the upper (respectively lower) Siegel half space.
Recall the Siegel threefold $X= \Sp_4(\Z) \backslash \H_2 =\GSp_4(\Z) \backslash \H_2^\pm$.
It can be viewed as the coarse moduli space of principally polarized  abelian surfaces over $\C$.

Let $V$ be the quadratic space
$$
V= \left\{ x:=x(a, b, c, d, r): = \kzxz {{}^t(J\tau)} {Ja} { -Jb} {J\tau}:\;   a , b \in \Q, \, \tau =\kzxz {c} {r} {r} {d} \in \Mat_2(\Q) \right\}
$$
with the quadratic form
$$
Q(x) = \frac{1}4 \tr(x^2) = ab + r^2  -cd = ab - \det \tau ,
$$
where $J=J_1$ with  $J_n =\kzxz {0} {I_n} {-I_n} {0}$.
We consider the even lattice
$$
L=\{x(a, b,c, d, r) : \;a , b, c, d, r \in \Z \}   \subset V.
$$
It is easily checked that $L'/L \cong \Z/2\Z$.
The group $\GSp_4$ acts on $V$ by conjugation and preserves the quadratic form.
This gives rise to an isomorphism $\Xi: \GSp_4 \stackrel{\sim}{\to} G=\GSpin(V)$ and an exact sequence
$$
1  \rightarrow \mathbb G_m \rightarrow \GSp_4 \rightarrow \SO(V) \rightarrow 1.
$$
Under this isomorphism, $\GSp_4(\hat\Z)$ is  the maximal compact open subgroup $K$ of $G(\A_f)=\GSpin(V)(\A_f)$ above.  According to \cite[Lemma~9.2]{YuPeng}, we have the following identification.

\begin{lemma}
\label{lem: Identification}   The isomorphism $\Xi$ induces an isomorphism
$$
\Xi:  \H_2^\pm  \stackrel{\sim}{\to}  \mathbb D , \quad  \Xi(\tau) = \kzxz {{}^t(J\tau)} {J \det(\tau) } { -J} {J\tau}
$$
which is $G$-equivariant. It also induces an  isomorphism
$$
\Sp_4(\Z) \backslash \H_2 \stackrel{\sim}{\to}  X_L .
$$
The inverse map  is given by
$$
\Xi^{-1}(x(a, b, c, d, r))= \frac{1}{b} \kzxz {c} {r} {r} {d}.
$$

\end{lemma}

\subsection{Principally polarized abelian surfaces with small CM}

Let $\mathbf A =(A, \lambda)$ be a principally polarized abelian surface over $\C$. It corresponds to a point $\tau = \kzxz {\tau_1} {\tau_{12}} {\tau_{12}} {\tau_2} \in \H_2$ in the sense that $\mathbf A \cong \mathbf A_\tau =(\C^2/\Lambda_\tau,  \lambda_\tau)$. Here $\Lambda_\tau =\tau \Z^2 + \Z^2$, and the Riemann form on  $\Lambda_\tau$  is given by $\lambda_\tau(x, y) = \Im( H_\tau(x, y))$ with
\begin{equation}
H_\tau(x, y) =  \, {}^t x (\Im \tau)^{-1} \bar y,
\end{equation}
which is  a  positive definite Hermitian form  on $\C^2$.  It is given explicitly by
 \begin{equation}
 \lambda_\tau(\tau x_1  +x_2,  \tau y_1 + y_2) = ({}^t x_1,  {}^t x_2) J_2  \begin{pmatrix} y_1 \\ y_2 \end{pmatrix}.
 \end{equation}
 This determines a principally polarized abelian surface $A_\tau=(\C^2/\Lambda_\tau, \lambda_\tau)$. We say that $\mathbf A$ has small CM by $\kay$ if $\Xi(\tau)$ is a small CM point with respect to $\kay$ on $X_L$.

\begin{proposition}
 \label{prop:smallCM}
Let $\mathbf A =(A,\lambda)$ with $A=\C^2/\Lambda$
be a principally polarized  abelian surface associated to $\tau \in \H_2$.   Then the following are equivalent:
\begin{enumerate}
\item $\mathbf A $ has small CM by $\kay$;

\item $\tau \in \H_2 \cap \Mat_2(\kay)$;

\item the ring
$$
\co=\co_\mathbf A =\{ r \in \kay:\,  r \Lambda \subset \Lambda\}
$$
is a quadratic order in $\kay$;

\item we have $\End^0(A) \cong \Mat_2(\kay)$;

\item $A$ is isogenous to the square $E^2$ of some elliptic curve $E$ with CM by $\kay$;

\item $A$ is isomorphic to the product $E_1 \times E_2$ of two elliptic curves with $\End^0(E_i) =\kay$.

\item $A$ is singular in the sense that the natural connecting map $H^1(A, \mathcal O_A^\times) \rightarrow H^{1,1} (A, \Z) = H^2(A, \Z)\cap H^{1, 1}(A, \C)$ is an isomorphism  up to torsion.

\end{enumerate}

\end{proposition}
\begin{proof} $(1) \Leftrightarrow (2)$   follows from the definition. $(2)\Leftrightarrow(3)$ follows from a direct calculation using $\Lambda =\Lambda_\tau$.

$(2) \Rightarrow (4)$: Assume $\mathbf A = \mathbf A_\tau$. We identify $\Lambda_\tau \otimes_\Z \Q \cong \kay^2$ naturally, and  $\C^2 = \kay^2 \otimes_\kay \C $.  Then
$$
\End(A) = \{ g \in \Mat_2(\C):  g \Lambda_\tau \subset \Lambda_\tau\},
$$ and so
$$
 \End^0(A)    = \{ g \in \Mat_2(\C):  g \kay^2 \subset \kay^2\} = \Mat_2(\kay).
$$

$ (4)  \Rightarrow (5)$:  Let $e_1=\kzxz {1} {0} {0} {0}$ and $e_2=\kzxz {0} {0} {0} {1}$ be the two standard idempotents of $\Mat_2(\kay)$.
Then $E_i =e_i A$ are CM elliptic curves,  and $E_1 \times E_2$ is isogenous to $A$. Moreover $E_1$ and $E_2$ are isogenous to each other as they both have CM by $\kay$.

The equivalence of (5), (6) and (7) are  \cite[Theorem 4.1]{MS}, see also \cite[Section 10.6]{BL}.
Finally, $(6) \Rightarrow (3)$ is obvious.
\end{proof}

Let $\kay$ be an imaginary quadratic field of  discriminant $d=d_\kay$. Then its ring of integers is given by $\co_\kay=\co_d$. In general, we write $\co_\delta$ for the quadratic order in $\kay$ with discriminant $\delta$.
Ignoring the principal polarization, T.~Mitani and  T.~Shioda proved the following result in \cite[Section 4]{MS} (see also \cite[Theorem 1]{Ka11} for a generalization).

\begin{proposition}
\label{prop:product}
Let $A= \C/\mathfrak a_1 \times \C/\mathfrak a_2$ and $B = \C/\mathfrak b_1 \times \C/\mathfrak b_2$ be products of  CM elliptic curves by the same imaginary quadratic field $\kay$ with quadratic lattices $\mathfrak a_i$ and $\mathfrak b_i$. Let
$$
\co_{\mathfrak a} =\{ r \in \kay:   r \mathfrak a \subset \mathfrak a\}
$$
 be the order of a quadratic lattice $\mathfrak a$ so that $\mathfrak a$ is a proper fractional ideal of $\co_{\mathfrak a}$.  Assume that $\co_{\mathfrak a_i}=\co_{e_i}$  and $ \co_{\mathfrak b_i}=\co_{f_i} $. Then $A \cong B$ if and only  if the following conditions hold:

\begin{enumerate}
\item $\mathfrak a_1 \mathfrak a_2  = c \mathfrak b_1 \mathfrak b_2$ for some $c \in \kay^\times$.

\item  $\operatorname{gcd}(f_1, f_2) = \operatorname{gcd}(e_1, e_2)$.

\item  $\operatorname{lcm}(f_1, f_2) = \operatorname{lcm}(e_1, e_2)$.
\end{enumerate}
\end{proposition}

Note that $\mathfrak a_1 \mathfrak a_2$ is a proper ideal of $\co_{\operatorname{gcd}(e_1, e_2)}$.
This proposition implies that  every abelian surface with small CM by $\kay$ has the `canonical' form
\begin{equation} \label{eq:canonical}
A =A(\mathfrak a, t) =\C/\co_{d_1} \times \C/\mathfrak a.
\end{equation}
Here $\mathfrak a$ is a quadratic lattice in $\kay$ with endomorphism ring $\co_{\mathfrak a} = \co_{d_2}$ with   $ d_i =d t_i^2$ and  $t_1 = t t_2$ for some positive integers $t, t_1, t_2$.
Notice that $t$,  and the ideal class $[\mathfrak a]\in \Cl(\co_{d_2})$ are complete invariants of $A$. We should have written the abelian surface as $A([\mathfrak a], t)$. However, we prefer to  denote it by $A(\mathfrak a, t)$ for simplicity.

We finish this subsection with a connection between abelian surfaces with small CM and the usual CM abelian surfaces.
Let $(E, \Phi=\{\sigma_1, \sigma_2\}) $ be a CM quartic field with CM type $\Phi$. Recall that a polarized  abelian surface $\mathbf A=(A, \lambda)$ over $\C$  is of CM type $(E, \Phi)$ if there is an embedding $\iota:  E \rightarrow \End^0(A)$ such that
\begin{enumerate}
\item[(i)] the Rosati involution induced by $\lambda$ restricts to the complex conjugation on $E$, and
\item[(ii)] there is a basis $\{ \omega_1,  \omega_2\}$ of  the differentials $\Omega_A$ such that $\iota(a)^* \omega_i = \sigma_i(a) \omega_i$ for $i =1, 2$.
\end{enumerate}
If  $E$ is not bi-quadratic, then $\iota$ is an isomorphism, and $\mathbf A$ corresponds to  a big CM point as in  \cite{YuPeng}. We are mainly interested in this paper the case that  $E$ is bi-quadratic. In this case,  \cite[Corollary 10.6.3]{BL} asserts that $A$ is isomorphic to a product of two isogenous CM elliptic curves. So $\mathbf A$ has small CM by some quadratic field $\kay$. The field $E$ has two imaginary quadratic subfields, and it is natural to ask {\it which one is our $\kay$?}
It  is  determined by the CM type $\Phi=\{\sigma_1, \sigma_2\}$ of $E$ associated to $\mathbf A$: $\kay $ is the reflex field of $\Phi$. Concretely,  we can write   $E =\Q(\sqrt D, \sqrt d) \hookrightarrow \C$  for $d <0 < D$ such that $\Phi=\{ \sigma_1 =1, \sigma_2= \sigma\}$ with  $\sigma(\sqrt D) = -\sqrt D$ and $\sigma(\sqrt d)=\sqrt d$. So the reflex field of $\Phi$, the field generated by $\norm_\Phi(z) =z \sigma(z), z \in E$, is $\Q(\sqrt d)$, and $\Q(\sqrt d)$ acts on $\Omega_A$ via the  natural embedding $\Q(\sqrt d) \subset \C$. Consequently,  $\Q(\sqrt d) = \kay$. So   we have proved the first part of the following proposition. We will prove the other two claims after Proposition  \ref{prop:Dagger} to avoid setting up a lot of notation now. In the proof, we will also give a method to find $D$ such that $\mathbf A$ has CM by $\Q(\sqrt d, \sqrt D)$ when  $\mathbf A$ has small CM  by $\Q(\sqrt d)$.

\begin{proposition}  \label{prop:CM} Let $E=\Q(\sqrt d, \sqrt D)$ be a  biquadratic CM number field with $d<0 <D$, and let $\Phi=\{1,  \sigma\}$ be the CM type of $E$ such that $\sigma(\sqrt d) = \sqrt d$. Let $\mathbf A=(A, \lambda)$ be a CM abelian surface over $\C$ of CM type $(E, \Phi)$. Then

(1) \quad $\mathbf A$ has small CM by $\kay$.

(2) \quad $\mathbf A$ is of CM type $(E_i, \Phi_i)$ for infinitely many biquadratic CM number fields $E_i=\Q(\sqrt{d}, \sqrt{D_i})$.

(3) Conversely,  every principally polarized  abelian surface $\mathbf A =(A, \lambda)$ with small CM by $\kay$ is of CM type $(E, \Phi)$ for some biquadratic CM number field $\kay \subset E$ and CM type $\Phi=\{1, \sigma\}$ as above.
\end{proposition}

By this proposition, we can see that the biquadratic CM field $E$ is not an invariant for a CM abelian surface $A$. The invariants of $A$ are $t$ and $[\mathfrak a]$ mentioned above. This is very different from the case where $E$ is not biquadratic.

\subsection{Galois action and small CM cycles}
\label{sect:smallCM}

Let $\mathbf A = (A=\C^2/\Lambda,  \lambda) $
 be a principally polarized abelian surface with small CM by $\kay$, and let $\co=\co_A $ be the order of $\kay$ which preserves $\Lambda$. We will also say that $\mathbf A$ has small CM by $\co_A$.  For any finite idele $t \in  T(\A_f)=\kay_f^\times$, let $n(t)$ be the positive rational number  such that  $t\bar t \hat{ \Z} =n(t) \hat\Z$.  We define the action of $\kay_f^\times$ on principally polarized abelian surfaces with small CM by $\kay$ as follows. Let $t\Lambda$ be the lattice in $\Lambda_\Q =\Lambda\otimes_\Z \Q$ such that $\widehat{t\Lambda} = t \hat{\Lambda}$. Alternatively,  $t\Lambda = t \hat{\Lambda}\cap \Lambda_\Q$. Next, we define the principal polarization $t\lambda$ on $t\Lambda$ by
\begin{equation}
t\lambda:  t\Lambda \times t\Lambda \rightarrow \Z, \quad t\lambda( x, y) = \lambda(x, y)/ n(t).
\end{equation}
Indeed, on $\widehat{t\Lambda} =t \hat\Lambda$,
$$
t\lambda(tx, ty) = \lambda(tx, ty)/ n(t) = \lambda(x, y) \frac{t \bar t}{n(t)} \in \hat\Z,
$$
which implies that $t\lambda$ is locally unimodular. It also implies that $t \lambda (x, y) \in \hat\Z \cap \Q=\Z$ for $x, y \in  t\Lambda$. This way, we have defined another principally polarized  abelian surface with small CM by $\kay$,
$$
t\mathbf A = ( tA= \C^2/t\Lambda, \,t\lambda).
$$
Notice that if $t \in \kay^\times$, then  multiplication by $t$ on $\C^2$ induces an isomorphism $ \mathbf A  \cong t \mathbf A$.
When  $t \in \hat{\co}_A^\times$, we have $t\Lambda =\Lambda$ and $t\lambda =\lambda$. So $t\mathbf A$ only depends on the image of $t$ in the class group $\kay^\times \backslash \kay_f^\times/\hat{\co}_A^\times$.

Denote for two positive integers $t$ and $t_2$,
$$
\mathcal L(t, t_2)= \{ \mathbf A=(A(\mathfrak a, t), \lambda):\;  [\mathfrak a]\in \Cl(\co_{d_2}), \, \lambda \hbox{ is a principal polarization on } A\}.
$$
Let $H_{d_1}$ be the ring class field of $\co_{d_1}$ (recall that $d_1 =d_2 t^2$), then $\Cl(\co_{d_1}) \cong \Gal(H_{d_1}/\kay) $ via the Artin map $t \mapsto \sigma_t$, or equivalently $\mathfrak a \mapsto \sigma_{\mathfrak  a}$. Via this identification, the group $\Gal(H_{d_1}/\kay) $ acts on  $\mathcal L(t, t_2)$ via
\begin{equation} \label{eq:GalAction}
\sigma_{\mathfrak a_1} \mathbf A = \left( \C/ \mathfrak a_1 \times \C/\mathfrak a_1 \mathfrak a, \,  \frac{\lambda}{\norm(\mathfrak a_1)}\right)\cong \left(\C/\co_{d_1}\times   \C/\mathfrak a_1^2 \mathfrak a, \,\lambda' \right),
\end{equation}
for some principally polarization $\lambda'$.  We choose one complex conjugation $\sigma$ on $H_{d_1}$ (a fixed preimage of the complex conjugation on $\kay$) which acts on $\mathcal L(t, t_2)$ via
\begin{equation} \label{eq:GalAction2}
\sigma (\mathbf A) = \bar{\mathbf A}= \left( \C/\co_{d_1} \times \C/\bar{\mathfrak a}, \,  \bar\lambda(x, y):=\lambda(\bar x, \bar y)\right).
\end{equation}
This gives an action of the Galois group $\Gal(H_{d_1}/\Q)  $ on $\mathcal L(t, t_2)$. Let $Z(\mathbf A)$ be  the orbit of $\mathbf A$ under the action of $\Gal(H_{d_1}/\Q)  $, or equivalently the formal sum of all Galois conjugations.

\begin{theorem}
\label{theo:Galois}
Let $\mathbf A =(A, \lambda) \in \mathcal A_2(\C)$ be a principally polarized abelian surface with small CM by $\kay$. Let $\tau \in \H_2$ be a corresponding point, and let
$U$ be the negative two plane in $V_\R$ associated to $\Xi(\tau)$ and $z_U^+ =\Xi(\tau)$.  Then  for any $t \in \kay_f^\times$ the translate $t \mathbf A$ corresponds to $[z_U^+, t]$ under the isomorphism $\Xi$, and $\bar{\mathbf A} =\sigma(\mathbf A)$ corresponds to  $\Xi(-\bar\tau)$. In  particular,
$Z(\mathbf A) = Z(\mathcal N)$ under the isomorphism $\Xi$, where $\mathcal N =U \cap L$.
\end{theorem}
\begin{proof} This follows directly from Shimura's reciprocity law. Indeed, the isomorphism $\Xi$ comes from the isomorphism
\begin{align*}
\Xi:  \GSp_4(\Q) \backslash \H_2^\pm \times \GSp_4(\A_f)/\GSp_4(\hat\Z)  &\cong  \GSpin(V)(\Q) \backslash  \mathbb D \times \GSpin(V)(\A_f)/K,
\\  [\tau, g] &\mapsto [\Xi(\tau),  \Xi(g)],
\end{align*}
where $\Xi: \GSp_4 \cong \GSpin(V) $ was given in  Section \ref{sect:siegel3fold}.  We are working on the Siegel modular threefold.  By Proposition \ref{prop:CM},  there is a CM type $(E=\Q(\sqrt d, \sqrt{D}), \Phi=\{1, \sigma\})$ such that $\mathbf A =(\C^2/\Lambda, \lambda)$ has CM by  $(E, \Phi)$. The reflex type  of $(E, \Phi)$ is  $(\kay, \tilde\Phi=\{1\})$. Now the Shimura reciprocity law \cite[Theorems 3.1, 3.2]{Ya16} shows that  $[\tau, t]$ gives $ t\mathbf A =(\C^2 /(t)\Lambda,  \lambda/n(t))$ as claimed, where $(t)$ is the $\co_{d_1}$-ideal generated by $t$.  In \cite{Ya16}, $\co$ is the full ring of the integers in $E$, the conclusion there also works for more general orders as in our case.
\end{proof}

\begin{corollary} \label{cor:DefinitionField} Let $\mathbf A =(A(\mathfrak a, t), \lambda) \in \mathcal A_2(\C)$ be a principally polarized abelian surface with small CM by $\co_{d_1}$. Then $\mathbf A$ is actually defined over $H_{d_1}$.
\end{corollary}


\begin{remark} The polarization plays a  critical  role in Theorem \ref{theo:Galois}. Without the polarization, the Galois group does not act faithfully on the product of CM  elliptic curves. For example, when $[\mathfrak a_1]^2 =1$ in the class group, then
\begin{align*}
\sigma_{\mathfrak a_1} (\C/\co_{d_1} \times \C/\mathfrak a)
 &= (\C/\mathfrak a_1 \times \C/ \mathfrak a_1 \mathfrak a) \cong (\C/\co_{d_1} \times \C/\mathfrak a).
\end{align*}

\end{remark}

\section{Lattices associated to abelian surfaces with small CM}
\label{sect:lattice}

\subsection{Set-up and natural construction of $L$} \label{sect:set-up}
From now on, we assume $\mathbf A =(A(\mathfrak a,t), \lambda)$ as in  (\ref{eq:canonical}) with $d_i=d t_i^2$ and $t_1 =t_2 t$. Write $\Lambda=\Lambda_\mathbf A = \co_{d_1} \oplus \mathfrak a\subset \kay^2$ as column vectors.   It is easy to see that
\begin{equation} \label{eq:EndA}
\End(A)= \zxz {\co_{d_1} } {\frac{t}{a}  \bar{\mathfrak a}}  {\mathfrak a} {\mathcal O_{d_2}} \subset \Mat_2(\kay).
\end{equation}
 Here   we write elements in  $\Lambda= \co_{d_1}  \oplus \mathfrak a$ as column vectors, and identify
 $$
 \frac{t}{a} \bar{\mathfrak a} = t \mathfrak a^{-1} \cong \hbox{Hom}_{\co} (\mathfrak a, \co),  \quad x \mapsto  l_{x}=\hbox{multiplication  by } x.
 $$
We  write
\begin{equation}
\label{eq:omega2}
\mathfrak a =\Z \omega_2  + \Z a, \quad  \omega_2 =\frac{b + \sqrt{d_2}}{2 }, \, a =\norm(\mathfrak a)>0, \,  c =\frac{b^2 -d_2}{4a} \in \Z.
\end{equation}
The quadratic form $Q(x) = x \bar x/a$ on $\mathfrak a$ is then given by $[a, b, c]= ax^2 + b xy + c y^2$.  Similarly, we write
\begin{equation}
\co_{d_1} = \Z \omega_1 + \Z 1,  \quad \omega_1 = \frac{b_1 + \sqrt{d_1}}2,  \,   a_1 =1, \hbox{ and }  c_1 = \frac{b_1^2 -d_1}{4a_1}.
\end{equation}

Later in this paper, we need to understand the lattice $\mathcal N =U\cap L$ and its complement $\mathcal P$ in $L$ for Schofer's  small CM value formula (Theorem  \ref{thm:SmallCM}), where $U$ is the negative two plane associated to $\Xi(\tau)$ when  we write $\mathbf A =A_\tau$. It turns out to be rather difficult to understand $\mathcal N$ (and its complement $\mathcal P$ in $L$)  this way. In this section, we will construct $L$, $\mathcal P$, and $\mathcal N$ directly from $\mathbf A$, and prove the necessary results about $\mathcal N$ and $\mathcal P$.
Recall that the principal polarization
$$
\lambda:  \Lambda \times \Lambda \rightarrow \Z
$$
is a unimodular symplectic form  on $\Lambda$ (i.e., the map is surjective) such that $\lambda(ix, x)$ is a positive definite quadratic form on $W=\C^2$.
For an endomorphism $f \in  \hbox{End}(A)$, we view it as an endomorphism  of $W$ which preserves $\Lambda$. Then its Rosati involution $f^\dagger$ is given, as an endomorphism of $W$,  by
\begin{equation} \label{eq:Rosati}
\lambda(f^\dagger(x), y) = \lambda(x, f(y)).
\end{equation}
Notice that the definition of the Rosati involution  does not depend on the complex structure on $V$ and depends only on the non-degenerate symplectic form on $\Lambda$. Define lattices
\begin{align}
\label{eq:L}
L_{\mathbf A}&= \{ f \in \hbox{End}(\Lambda):\;  f^\dagger =f ,\,  \tr (f) =0\} ,
\\
\mathcal P_{\mathbf A}  &= \{ f \in  \hbox{End}(A):\; f^\dagger =f ,\,  \tr (f) =0\}  \subset L_{\mathbf A}  \notag
\end{align}
equipped with the quadratic form $Q(f) = \frac{1}4 \tr (f^2) =  \frac{1}4 \tr (f f^\dagger) $. Here the trace $\tr$ is taken as the trace of an endomorphism of the rank $4$ lattice $\Lambda$. Using the terminology in  \cite[Section~2]{KRSiegel}, $\mathcal P_{\mathbf A}$ is the lattice of special endomorphisms   of $(A, \lambda)$.  In particular, \cite[Lemma~2.2]{KRSiegel} asserts that  $f^2 =Q(f)  \id$ for $f \in  \mathcal P_{\mathbf A}$.  Let
\begin{equation} \label{eq:N}
\mathcal N_{\mathbf A} =\{ x \in L_{\mathbf A}: \; (x, y) =0  \hbox{ for all } y \in \mathcal P_{\mathbf A}\}
\end{equation}
be the orthogonal complement of $\mathcal P_{\mathbf A}$ in $L_{\mathbf A}$.  We will also need later
\begin{align}
\tilde L_{\mathbf A} &= \{ f \in \hbox{End}(\Lambda):\;  f^\dagger =f \} ,
\\
\tilde{\mathcal P }_{\mathbf A}  &= \{ f \in  \hbox{End}(A):\; f^\dagger =f   \}  \subset  \tilde L_{\mathbf A} . \notag
\end{align}

Let $\{ e_1, \dots, e_4\}$ be a  basis of $\Lambda$ with  Gram matrix  $J_2$, and identify $\alpha \in \End(\Lambda)$ with $\alpha =\kzxz {A} {B} {C} {D} \in \Mat_4(\Z)$  via
$$
(\alpha(e_1), \dots,  \alpha(e_4)) =( e_1, \dots, e_4) \begin{pmatrix} A & B \\ C &D\end{pmatrix}.
$$

\begin{lemma}
\label{lem:Iden}
Let the notation be as above.
\begin{enumerate}
\item[(1)]
For $\alpha = \kzxz {A} {B} {C} {D} \in \Mat_4(\Z) = \End(\Lambda)$, we have
 $$
\alpha^\dagger = J_2^{-1} \,  {}^t\alpha \,J_2=
\kzxz  {{}^tD} {- {}^tB} {-{}^tC} { {}^tA}.
 $$
\item[(2)]
The quadratic lattice $(\tilde L_{\mathbf A}, Q)$ is unimodular  of signature $(4, 2)$ (but not even) and is independent of $\mathbf A$ (up to isomorphism).  In fact,
$$
\tilde L_{\mathbf A} \cong \tilde L:=\{x= \kzxz {{}^tA} {J_2 a} {-J_2 b} {A}:\; a, b \in \Z, \, A=\kzxz {r} {c} {d} {s} \in \Mat_2(\Z)\}
$$
with quadratic form $Q(x) = \frac{1}4 \tr(x^2) =\frac{1}2 \tr (A^2) + ab = \frac{r^2+s^2}2 + ab+ cd$.
\item[(3)]
The quadratic lattice $(L_{\mathbf A}, Q)$ is isomorphic to the  quadratic lattice  defined in Section~\ref{sect:siegel3fold}, that is,
  $$
L_{\mathbf A} \cong   L=\{ x(a, b,c,d,r)=  \kzxz {{}^t(J\tau)} {Ja} { -Jb} {J\tau}:\;   a , b \in \Z,\,  \tau =\kzxz {c} {r} {r} {d} \in \Mat_2(\Z) \}
$$
with the quadratic form
$$
Q(x) = \frac{1}4 \tr(x^2) = ab + r^2  -cd = ab - \det \tau .
  $$
\item[(4)]
The lattice $\mathcal N_{\mathbf A}$ is also the orthogonal complement of $\tilde{\mathcal P}_{\mathbf A}$ in  $\tilde L_{\mathbf A}$.
\end{enumerate}
\end{lemma}

\begin{proof}
Claims (1)--(3) follow from a direct calculation which we leave to the reader.

To prove (4),  notice that the orthogonal complement $L_{\mathbf A}^\perp$ of $L_{\mathbf A}$ in $\tilde L_{\mathbf A}$ is $\{a I_4:\, a \in \Z\}$. Hence
$$
\mathcal P_{\mathbf A} \oplus L_{\mathbf A}^\perp \subset \tilde{\mathcal P}_{\mathbf A}
$$
is of finite index.
Since $\mathcal N_{\mathbf A}$ is contained in $L_{\mathbf A}$ and is the orthogonal complement of $\mathcal P_{\mathbf A}$ in $L_{\mathbf A}$, it is also the orthogonal complement of $\tilde{\mathcal P}_{\mathbf A}$ in $\tilde L_{\mathbf A}$.
\end{proof}

From now on, we will identify the  lattices $L_{\mathbf A}$ with the fixed lattice $L$ via the isomorphism in Lemma~\ref{lem:Iden} (3), and analogously the $\tilde L_{\mathbf A}$ with  $\tilde L$.  This also identifies the lattices $\mathcal P_{\mathbf A}$ and $\mathcal N_{\mathbf A}$ as sublattices of $L$, and analogously the $\tilde{\mathcal P}_{\mathbf A}$ as  sublattices of $\tilde L$. We will omit subscript $\mathbf A$ when it is clear from the context.
The following lemma   follows  directly from \cite[Section~2]{KRSiegel} (essentially  Albert's classification of abelian surfaces).

\begin{lemma}\label{lem:KR} Let $(A, \lambda)$ be a principally polarized abelian surface. Let $e_1, \dots, e_r$ be a $\Z$-basis of the positive definite quadratic lattice $\mathcal P$, where $0 \le r \le 3$ is the rank of  $\mathcal P$.
Then  $(A, \lambda)$ belongs to the sub-Shimura variety associated to the orthogonal complement of $\mathcal P$. In particular, if  $\calP$ has rank $3$ (maximal rank), then  $(A, \lambda) = z_\mathcal N^+$ is a small CM point associated to the negative definite lattice  $\mathcal N=\mathcal P^\perp$ in $L$.
\end{lemma}

In the next two subsections, we will show that the converse holds, i.e., if $(A, \lambda)=z_{\mathcal N}^+$ has small CM by $\kay$, then $\mathcal P$ has rank $3$, and $\mathcal N=\mathcal P^\perp $. Furthermore, we will prove some results on the lattices $L$, $\mathcal P$, and $\mathcal N$ associated to $(A, \lambda)$ which are needed later.

\subsection{The product polarization}
We start with a very special and simple case: the product polarization $\lambda_0 =\lambda_1 \times \lambda_2$. We use $\lambda_0$ to distinguish it from general principal polarizations.
Recall that the product polarization  $\lambda_0$ on $\Lambda$  given by
 \begin{equation}
 \lambda_0(x, y) = \tr_{\kay/\Q} \frac{x_1 \bar{y}_1}{\sqrt{d_1}} + \tr_{\kay/\Q} \frac{x_2 \bar{y}_2}{a \sqrt{d_2}}.
 \end{equation}
 We renormalize  the associated positive Hermitian form as (replacing $i$ by $\sqrt{d_1}$ so that $H_0(x, y) \in
 \kay$ for $x, y \in \Lambda$)
 \begin{equation}  \label{eq:Hermitian}
 H_0(x, y) = \lambda_0(\sqrt{d_1} x, y) + \sqrt{d_1}  \lambda_0(x, y)= 2 x_1 \bar{y}_1+ \frac{2 t x_2 \bar{y_2}}{ a }.
 \end{equation}
When  $\alpha \in \End(A)$, its Rosati involution is also characterized by
\begin{equation}
H_0(\alpha x, y) = H_0(x, \alpha_0^\dagger y).
\end{equation}
A direct calculation gives the following lemma.

\begin{lemma} Let the notation be as above. For $\alpha = \kzxz {\alpha_1} {\frac{t}{a} \alpha_2} {\alpha_3} {\alpha_4} \in \End(A) \subset \Mat_2(\kay)$, we have
$$
\alpha^\dagger_0 = \kzxz {\bar{\alpha}_1} {\frac{t}{a} \bar{\alpha}_3} {\bar{\alpha}_2} {\bar{\alpha}_4}.
$$
Moreover,
\begin{enumerate}
\item
$$
\tilde{\mathcal P}_0 =\{ \kzxz {\alpha_1} {\frac{t}{a} \bar{\beta}} {\beta} {\alpha_2}:\;  \alpha_i \in \Z,  \beta \in \mathfrak a\}
$$
has the quadratic form  $Q(\alpha) = \frac{\alpha_1^2+ \alpha_2^2}2 + \frac{t}a \norm(\beta)$.  In particular,
$$
\tilde{\mathcal P}_0'/\tilde{\mathcal P}_0 \cong   \left(\frac{1}{\sqrt{d_1}}\mathfrak a/\mathfrak a,  \frac{t}a \norm \right), \quad |\tilde{\mathcal P}_0'/\tilde{\mathcal P}_0|=|d_1|.
$$

\item
$$
\mathcal P_0 =\{ \kzxz {\alpha} {\frac{t}{a} \bar{\beta}}  {\beta} {-\alpha}:\,  \alpha \in \Z,  \beta \in \mathfrak a\}
$$
has the quadratic form  $Q(\alpha) = \alpha^2 +  \frac{t}a \norm(\beta)$, and
$$
\mathcal P_0'/\mathcal P_0 \cong   (\frac{1}2\Z/\Z \oplus  \frac{1}{\sqrt{d_1}}\mathfrak a/\mathfrak a,  x^2 + \frac{t}a \norm(\beta)),  \quad |\mathcal P_0'/\mathcal P_0|= 2|d_1|.
$$
\end{enumerate}
\end{lemma}

 To compute $\tilde L_0$ and $\mathcal N_0$, choose the standard basis of $\Lambda$
 $$
 e_1 =\,   {}^t(\omega_1, 0), \, e_2=\,{}^t(0, \omega_2),  e_3 =\,  {}^t(1, 0). \hbox{ and } e_4 = \,  {}^t(0, a),
 $$
 with Gram matrix $J_2$. With respect to this basis, we identify $\End(\Lambda) =\Mat_4(\Z)$.
 So Lemma~\ref{lem:Iden} applies here. We will use subscript $0$ to indicate the dependence on the polarization $\lambda_0$,

With respect to the chosen standard basis above,  the natural  embedding $\iota:  \End(A) \rightarrow \End(\Lambda)$ is given by the following lemma.

\begin{lemma}
\label{lem:embedding}
For $\alpha =\kzxz{x}{y}{z}{w} \in \End(A) \subset \Mat_2(\kay)$, let $\iota(\alpha) \in \Mat_4(\Z) $ be  the matrix of $\iota(\alpha)$ with respect to the basis $\{ e_1, e_2, e_3, e_4\}$. Write ($a_1 =1$ and $a_2 =a$ in our notation)
\begin{align*}
(x\omega_1, x a_1) &= (\omega_1, a_1) \kzxz {x_1} {x_2} {x_3} {x_4},   \quad  (w \omega_2, w a_2) =  (\omega_2, a_2) \kzxz {w_1} {w_2} {w_3} {w_4},
\\
(z \omega_1, z a_1) &= (\omega_2, a_2) \kzxz {z_1} {z_2} {z_3} {z_4},   \quad  (y\omega_2, ya_2) =  (\omega_1, a_1) \kzxz {y_1} {y_2} {y_3} {y_4}.
\end{align*}
Then  we have
\begin{equation} \label{eq:alpha}
\iota(\alpha)
=
 \begin{pmatrix}
   x_1 & y_1 &x_2 &y_2
   \\
   z_1 & w_1 & z_2 & w_2
   \\
   x_3 & y_3 &x_4 & y_4
   \\
   z_3 &w_3 &z_4 &w_4
 \end{pmatrix}.
\end{equation}
In  particular
$$
\iota(\tilde{\mathcal P}_0) =\{ \iota(\kzxz {\alpha_1} {\frac{t}a \bar{\beta}} {\beta} {\alpha_2} )
=\begin{pmatrix}
 \alpha_1 & z_4 & 0 & -z_2
   \\
   z_1 &\alpha_2  & z_2 &0
   \\
  0 & - z_3 &\alpha_1 & z_1
   \\
   z_3 &0 &z_4 &\alpha_2
 \end{pmatrix}:\,  \alpha_1, \alpha_2 \in \Z,  \beta \in  \mathfrak a\}.
$$
Here the $z_i$ are given by
$$
(\beta \omega_1,  \beta a_1) = ( w_2, a_2) \kzxz {z_1} {z_2} {z_3} {z_4}.
$$
So $\iota(\alpha) \in  L_0$ if and only if $\alpha_2 =-\alpha_1$.
\end{lemma}
We will identify $\alpha$ with $\iota(\alpha)$ and drop $\iota$ from the notation for simplicity when there is no risk of confusion.  By Lemma \ref{lem:embedding}, $\mathcal P_0 \subset L_0$ has the  basis
\begin{align*}
p_1&= \iota( \kzxz {1} {0} {0} {-1} ) =\hbox{Diag}(1,-1, 1, -1),
\\
p_2 &=\iota(\kzxz {0} {t} {a} {0} ) = \begin{pmatrix}
{0}  &{1} &{0} &{0}
\\
{ta} &{0} &{0} &0
\\
0 &-(b_1 -t b)/2  &0 &ta
\\
(b_1 -t b)/2 &0 &1 &0
\end{pmatrix},
\\
p_3 &=\iota(\kzxz {0} {\frac{t}a \bar{\omega}_2} {\omega_2} {0})
 =\begin{pmatrix}
 0 &0 &0 & -1
 \\
 (b_1 + t b)/2 &0 &1 &0
 \\
 0 & tc &0 &(b_1 + t b)2
 \\
 -t c &0 &0 &0
 \end{pmatrix}.
 \end{align*}
The lattice $\tilde{\mathcal P}_0$  has a basis $\{ p_0= \hbox{Diag}(1, 0, 1, 0), p_1,  p_2, p_3\}$.

\begin{lemma}
\label{lem:N0}
Let the notation be as above. Then
$$
\mathcal N_0 =\left\{n=\begin{pmatrix}
0 &x &0 &y
\\
w &0 &-y &0
\\
0 &z &0 &w
\\
-z &0 &x &0
\end{pmatrix}:\;
 z = tc y -  \frac{b_1+ tb}2  x , \quad w = -  \frac{b_1 - t b}2 y   -  ta x\right\}
$$
has the basis $ n_1, n_2$ with
$$
n_1= \begin{pmatrix}
0 &1 &0 &0
\\
-ta &0 &0 &0
\\
0  & -\frac{b_1 +t b}{2} &0 & -ta
\\
\frac{b_1 +t b}{2} &0 &1 &0
\end{pmatrix},
\hbox{ and }
n_2 = \begin{pmatrix}
0 &0 &0 &1
\\
-\frac{b_1-t b}{2}  &0 &-1 &0
\\
0 &tc &0 &-\frac{b_1-t b}{2}
\\
-tc &0 &0 &0
\end{pmatrix}.
$$
The Gram matrix of $\calN_0$ with respect to the basis $n_1, n_2$ is  $\kzxz {-2ta} {tb} {tb} {-2tc} $, i.e.,
$$
\left(\mathcal N_0,  Q(n) =\frac{1}4 \tr n^2\right) \cong \left(\bar{\mathfrak a},  Q(x) =- \frac{t}a x \bar x\right).
$$
\end{lemma}

\subsection{General principal polarizations}

\begin{lemma}  \label{lem:dual} Let $ A= A(\mathfrak a, t)$ be as (\ref{eq:canonical}), and let    $ \lambda_0$ be the product polarization on $A$ with associated positive definite Hermitian form $H_0$ defined in (\ref{eq:Hermitian}).  Let $\lambda$ be another  principal polarization on $A$ with associated Hermitian form   $H$. Then there is an $\alpha =\kzxz {\alpha_1} {\frac{t}a \bar \beta }{ \beta} {\alpha_2}  \in  \tilde{\mathcal P}_0$ such that $\alpha_i \in \Z_{>0}$ and $\det \alpha =1$  and
$$
H(x, y) = H_0(\alpha(x), y),
$$
or equivalently $\lambda(x, y) = \lambda_0(\alpha(x), y)$. The converse is also true: for $\alpha$  as above, $\lambda =\lambda_0\circ \alpha$  gives a principal polarization.
\end{lemma}

\begin{proof}
This is well-known and is true in greater generality. We sketch a proof here for completeness. Write $W =\C^2$  and $ \Lambda = \mathcal O_{d_1} \oplus \mathfrak a$.  We identify the dual of $A$ as $A^\vee =W^\vee/\Lambda^\vee $, where $W^\vee = \Hom_\C(W, \bar\C)$ and
$$
\Lambda^\vee =\{ f \in W^\vee:\,  \hbox{Im}(f/\sqrt{d_1}) (\Lambda) \subset \Z\}.
$$
Then  $H$ (respectively $H_0$) gives the polarization
$$
\lambda:  A \rightarrow A^\vee, \quad \lambda_H(x) (y) =H(x, y).
$$
So $\lambda = \lambda_0 \circ \alpha$ means that $H(x,y) =H_0(\alpha(x), y)$ with   $\alpha \in \GL(V)$, and $\alpha (\Lambda)  = \Lambda$. Furthermore,
$\overline{H(x, y)} = H(y, x)$ if and only if $\alpha^\dagger_0=\alpha$. The fact that  $\lambda$ and $\lambda_0$ are both principal polarizations  implies that $\det \alpha =1$ and $\alpha_1, \alpha_2 >0$ as expected.
\end{proof}

\begin{proposition}
\label{prop:Dagger}
Let $\alpha =\kzxz {\alpha_1} {\frac{t}a \bar \beta }{ \beta} {\alpha_2}  \in  \tilde{\mathcal P}_0$ such that $\alpha_i >0$ and $\det \alpha =1$. Let $\lambda=\lambda_0 \circ \alpha$ and $H=H_0 \circ \alpha$ be the associated principal polarization on $A=A( \mathfrak a, t)$.   Then
 we have the following   (the subscript $\lambda$ indicates the dependence on $\lambda$).
\begin{enumerate}

\item  For $g \in \End(\Lambda)$,
 \begin{equation} \label{eq:Dagger}
g^\dagger_\lambda = \alpha^{-1} g_0^\dagger \alpha.
\end{equation}
\item
$$
\tilde L_\lambda  = \alpha^{-1} \tilde L_0 = \tilde L_0 \alpha .
$$

\item
$$
\tilde{\mathcal P}_{\lambda} = \alpha^{-1} \tilde{\mathcal P}_0 = \tilde{\mathcal P}_0 \alpha, \quad  \mathcal P_\lambda= \tilde{\mathcal P}_\lambda^{\tr =0}.
$$
and

\item
$$
\mathcal N_{\lambda} =  \alpha^{-1} \mathcal N= \mathcal N_0 \alpha .
$$
\end{enumerate}
\end{proposition}
\begin{proof}
For $g \in \hbox{End}(\Lambda)$, $x, y \in \Lambda$,  we have
$$
\lambda( g x, y) = \lambda_0(\alpha g(x),  y) = \lambda_0(x,   (\alpha g)^\dagger_0 P (y))= \lambda_0(x,  g^\dagger_0 \alpha y),
$$
and
$$
\lambda(x, g^\dagger_\lambda y) = \lambda_0( \alpha x,   g^\dagger_\lambda y) = \lambda_0(x,  \alpha  g^\dagger_\lambda y).
$$
So (\ref{eq:Dagger}) is true.
Now (2) and (3)  are  clear. Finally, (4) follows from
$$
(\alpha^{-1} g_1, g_2 \alpha) = \frac{1}4 \tr \alpha^{-1} g_1  g_2 \alpha =  \frac{1}4 \tr g_1 g_2
$$
for $g_1 \in \tilde{\mathcal P}_0$ and $g_2 \in \mathcal N_0$.
\end{proof}

Notice that $L_\lambda \ne L_0 \alpha$ and $ \mathcal P \ne \mathcal P_0 \alpha$ in general.

\begin{proof}[Proof of Proposition \ref{prop:CM}]
We are now ready to prove (2) and (3) of Proposition~\ref{prop:CM}.
When  $\lambda=\lambda_0$ is the product polarization, clearly $\iota: \kay \rightarrow \Mat_2(\kay)$, $\iota(r) = r I_2$ satisfies $\iota(r)^\dagger_0 =\iota(\bar r)$, and $\iota(\kay) $ is in the center of $\Mat_2(\kay)$. For general $\lambda$, (\ref{eq:Dagger}) implies that the above map still satisfies $\iota(r)^\dagger_\lambda =\iota(\bar r)$ for all $r \in  \kay$. Now for any $g \in \mathcal P_\lambda$, we have $g^2= D I_2$ for some $D \in \Z_{>0}$ by \cite[Lemma 2.2]{KRSiegel}. So the map
$$
\iota: E=\Q(\sqrt D,  \sqrt d) \rightarrow \Mat_2(\kay),  \quad \iota(\sqrt d) =\sqrt d I_2, \, \iota(\sqrt D) =g
$$
makes $\mathbf A=(A, \lambda)$ a CM abelian surface of CM type $( E, \Phi)$.  Clearly there are infinitely many such square-free $D$. This proves (2) and (3) of Proposition \ref{prop:CM}.
\end{proof}

A main technical result of this paper is the following proposition.

\begin{proposition} \label{prop:N} Let $\mathbf A =(A(\mathfrak a, t),  \lambda)$   be a principally polarized  CM  abelian surface with $\lambda =\lambda_0 \circ \alpha$ as above and $\alpha =\kzxz {\alpha_1} {\frac{t}a \bar \beta }{ \beta} {\alpha_2}  \in  \tilde{\mathcal P}_0$,  and $\det \alpha =1$.  Then
\begin{enumerate}
\item  we have $(\mathcal N,  Q) \cong (\bar{\mathfrak a},  - \frac{t}a x \bar x)$, and
$\mathcal N'/\mathcal N \cong ( \frac{1}{\sqrt{d_1}} \bar{\mathfrak a}/ \bar{\mathfrak a}, \,x\mapsto -\frac{t}a x \bar x)$;

\item   we have  $\tilde{\mathcal P}'/\tilde{\mathcal P}   \cong (\frac{1}{\sqrt{d_1}} \bar{\mathfrak a}/ \bar{\mathfrak a}, x\mapsto \frac{t}a x \bar x) $;

\item when  $d_1$ is odd,  we have
$$
{\mathcal P}'/{\mathcal P}   \cong \left(\frac{1}2\Z/\Z \oplus \frac{1}{\sqrt{d_1}} \bar{\mathfrak a}/ \bar{\mathfrak a},\, (y, x) \mapsto  y^2+ \frac{t}a x \bar x\right).
$$

 \end{enumerate}
\end{proposition}
\begin{proof}
 (1)\quad By Proposition \ref{prop:Dagger},  $\mathcal N$ has a basis given by $n_j'=n_j\iota(\alpha)$ where $\{n_1,n_2\}$ is the basis of $\mathcal N_0$ given by Lemma \ref{lem:N0}.
A direct and tedious calculation yields
\begin{eqnarray*}
(n_1', n_1')= \tr(n_1'n_1')/2&=&-2\alpha_1\alpha_2ta+2(ab\beta_1\beta_2+a^2\beta_2^2+a\beta_1^2c)t^2\\
&=&-2(\alpha_1\alpha_2-1)ta-2ta+\frac{t^2}{2}\left((b\beta_1+2a\beta_2)^2+(4ac-b^2)\beta_1^2\right)\\
&=&-2ta+\frac{t^2}{2}\left(-4\norm(\beta)+\tr(\beta)^2-d_2\beta_1^2\right)\\
&=&-2ta .
\end{eqnarray*}
Similarly, we have
$$
(n_1', n_2') =\frac{1}2 \tr(n_1'n_2') = tb  \quad \hbox{ and } \quad
(n_2', n_2') =\frac{1}2  \tr(n_2'n_2') =-2tc.
$$
Therefore, the Gram matrix of $\{n_1',n_2'\}$ is
$$B=(\frac{1}2 \tr(n_i'n_j'))_{1\leq i,j\leq2}=\mat{-2ta&tb\\tb&-2tc}$$
with $\det(B)=t^2(4ac-b^2)=-t^2d_2=-d_1$. This proves (1).

(2) \quad  Notice that $\tilde L $ is unimodular, and  $|\tilde L'/\tilde L|=1$.  Applying Proposition \ref{prop:fqm2} to $(\tilde L, \mathcal N)$, we obtain $(2)$.

(3) \quad Since $2 \nmid d _1$,  the quantity $|L'/L|=2$ is prime to $|\mathcal N'/\mathcal N|=d_1$. Applying Proposition \ref{prop:fqm2} to $(L, \tilde N)$ gives the exact sequence
$$
1 \rightarrow L'/L\cong   \frac{1}2 \Z/\Z  \rightarrow \mathcal P'/\mathcal P  \rightarrow  \mathcal N'/\mathcal N \cong \frac{1}{\sqrt{d_1}} \bar{\mathfrak a}/ \bar{\mathfrak a} \rightarrow 0.
$$
Since $(2, d_1) =1$, the exact sequence is split, and
$$
\mathcal P'/\mathcal P \cong  \frac{1}2 \Z/\Z \oplus  \frac{1}{\sqrt{d_1}} \bar{\mathfrak a}/ \bar{\mathfrak a}.
$$
The claim about the  quadratic form follows from  Remark \ref{rem:quad}.
\end{proof}

It is natural to  ask when two polarizations on the same abelian surface are equivalent, i.e.,  when there is some $f \in \Aut(A)$ with  $\lambda' = f^\vee \lambda f$.  This was addressed in  \cite[Proposition 3.1]{GHR} when $A=E \times E$. The same argument gives the following proposition whose proof we leave to the reader.

\begin{proposition} Let the notation be as in  Proposition \ref{prop:N}. Then two polarizations $\lambda =\lambda_0 \circ \alpha$ and $\lambda'= \lambda_0 \circ \alpha'$ are equivalent if and only if  there exists a matrix $P =\kzxz {p_1} {\frac{t}a \bar{p}_3} {p_4} {p_2} \in\Aut(A)$  as in (\ref{eq:EndA}) (i.e.,  $p_i \in \co_{d_i}$ for $i=1, 2$, and $p_3, p_4 \in \mathfrak a$) such that $P_0^\dagger\alpha P=\alpha'$.
\end{proposition}

\subsection{Appendix: Primitive sublattices and discriminant forms}
\label{sect:app}

Here we recall some results about primitive sublattices of even lattices, their orthogonal complements, and the associated discriminant forms. In particular we slightly generalize Proposition 1.2 of \cite{Ebeling} to the non-unimodular setting.

Let $(L,Q)$ be an even lattice, that is, $L$ is a free $\Z$ module of rank $r$ equipped with a $\Z$-valued non-degenerate quadratic form $Q$. We briefly write $L_\Q=L\otimes_\Z\Q$, and if $p$ is a prime we denote by $L_p=L\otimes_\Z \Z_p$ the associated lattice over the $p$-adic integers.

The integral bilinear form associated with $Q$ is given by $(x,y)= Q(x+y)-Q(x)-Q(y)$. The dual lattice
\[
L'=\{ x\in L_\Q\mid \; \text{$(x,y)\in \Z$ for all $y\in L$}\}
\]
contains $L$ with finite index, and the quotient $L'/L$ is called the discriminant group of $L$.
The quadratic form $Q$ induces a $\Q/\Z$-valued quadratic form on $L'/L$. The discriminant group can be written as the orthogonal sum of its $p$-components,
\[
L'/L = \bigoplus_{\text{$p$ prime}} L_p'/L_p.
\]
For a non-zero integer $d$ we put
\begin{align}
\label{eq:defLS}
(L'/L)_d = \bigoplus_{\text{$p\mid d$ prime}} L_p'/L_p, \qquad (L'/L)^d = \bigoplus_{\text{$p\nmid d$ prime}} L_p'/L_p,
\end{align}
so that $L'/L = (L'/L)_d \oplus (L'/L)^d $.

A sublattice $N\subset L$ is called primitive, if $N= N_\Q\cap L$, or equivalently, if the quotient $L/N$ is free. Then there exists a section $s:L/N\to L$ to the quotient map
$L\to L/N$ such that
\[
L= N\oplus s(L/N),
\]
where the direct sum is in general not orthogonal. This implies that the natural restriction map $\Hom(L,\Z)\to \Hom(N,\Z)$ is surjective.

Let $N\subset L$ be a primitive sublattice, and write $N'\subset N_\Q\subset L_\Q$ for the dual of $N$. It is easily seen that the orthogonal projection $L_\Q\to N_\Q$ restricts to a map
\[
\pi_N:L'\to N'.
\]
This map is surjective: If $y\in N'$ then it induces a homomorphism $N\to \Z$, $x\mapsto (y,x)$. By the previous paragraph, this is the restriction of a homomorphism $L\to \Z$, and hence there exists a $z\in L'$ such that $(y,x)= (z,x)$ for all $x\in N$. This means that $\pi_N(z)=y$.
The kernel of $\pi_N$ is equal to
\[
\tilde N:= L'\cap N_\Q^\perp.
\]
Hence we have an exact sequence
\[
0\to \tilde N \to L' \to N'\to 0.
\]
By restriction it induces an exact sequence
\[
0\to \tilde N \to \tilde N\oplus N \to N\to 0,
\]
and hence an isomorphism of finite abelian groups
\begin{align}
\label{eq:fqm1}
\bar\pi_N:L'/(\tilde N\oplus N)\to N'/N.
\end{align}

We denote by $N^\perp := L\cap N_\Q^\perp$ the orthogonal complement of $N$ in $L$. The inclusion $N^\perp\to \tilde N$ induces an exact sequence
\begin{align}
\label{eq:fqm2}
0\to \tilde N/N^\perp \to L'/(N^\perp \oplus N)\to L'/(\tilde N\oplus N) \to 0.
\end{align}
Moreover, the inclusion $\tilde N\to L'$ induces an injective map
\begin{align}
\label{eq:fqm3}
\tilde N/N^\perp \to L'/L.
\end{align}

\begin{proposition}
\label{prop:fqm1}
Let the notation be as above.

i)
If $|N'/N|$ is coprime to $|L'/L|$, then the homomorphism \eqref{eq:fqm3} is an isomorphism.

ii) If $|(N^\perp)'/N^\perp|$ is coprime to $|L'/L|$, then $\tilde N
 = N^\perp$ and \eqref{eq:fqm3} is trivial.
\end{proposition}

\begin{proof}
i) We have to show that \eqref{eq:fqm3} is surjective. It suffices to show this locally, i.e., that for every prime $p$ the induced injective map of finite $\Z_p$-modules
\begin{align}
\label{eq:fqm4}
\tilde N_p/N_p^\perp \to L_p'/L_p
\end{align}
is surjective.

If $p$ is a prime dividing $|N'/N|$, then $p$ does not divide  $|L'/L|$. Hence $L_p'=L_p$ and the surjectivity of \eqref{eq:fqm4} is trivial.

If $p$ is a prime not dividing $|N'/N|$, then $N_p$ is unimodular over $\Z_p$. Hence there are  orthogonal direct sum decompositions
\begin{align*}
L_p&=N_p^\perp\oplus N_p,\\
L_p'&=(N^\perp)'_p\oplus N_p.
\end{align*}
This also implies $L_p'= \tilde N_p\oplus N_p$ and therefore the surjectivity of \eqref{eq:fqm4}.

ii) As for (i) we argue locally and show that for every prime $p$ we have the equality $N^\perp_p =\tilde N_p$.

If $p$ is a prime dividing $|(N^\perp)'/N^\perp|$, then $L_p'=L_p$ as above. The assertion follows from the injectivity of \eqref{eq:fqm4}.

If $p$ is a prime not dividing $|(N^\perp)'/N^\perp|$, then $(N^\perp)_p$ is unimodular over $\Z_p$.
We obtain an orthogonal splitting $L_p= (N^\perp)_p \oplus N_p$ as above and also
$L_p'= (N^\perp)_p \oplus N_p'$. This implies $\tilde N_p=N_p$, concluding the proof of the proposition.
\end{proof}

We now employ the proposition to construct a map of discriminant forms. Let $L$ be an even lattice and $N\subset L$ a primitive sublattice as above. Then the lattice $P:= N^\perp$ is also a primitive sublattice of $L$. Since $N=P^\perp$, we may also apply  the above arguments for $P$ instead of $N$.

Assume that $|N'/N|$ is coprime to $|L'/L|$. Then Proposition \ref{prop:fqm1} (i) implies that
$\tilde N/P\cong L'/L$, and Proposition \ref{prop:fqm1} (ii)
implies $\tilde P = P^\perp =N$.
The exact sequence \eqref{eq:fqm2} gives rise to a commutative diagram with exact rows, where all vertical arrows are isomorphisms:
\begin{align*}
\xymatrix{
0\ar[r] & \tilde N/P \ar[r]\ar[d]  & L'/(P\oplus N) \ar[r]\ar[d]^{\bar\pi_P}   & L'/(\tilde N\oplus N)\ar[r] \ar[d]^{\bar\pi_N}& 0\\
0\ar[r] & L'/L\ar[r] & P'/P  \ar[r]   & N'/N \ar[r] & 0 .
}
\end{align*}
This implies the following result.

\begin{proposition}
\label{prop:fqm2}
Let the notation be as above and assume that $|N'/N|$ is coprime to $|L'/L|$.
Then the map
\[
g:P'/P\to N'/N, \quad x \mapsto g(x):= \bar\pi_N\big(\bar\pi_P^{-1}(x)+(\tilde N+N)\big)
\]
defines a surjective homomorphism  with kernel $L'/L$.
In particular, we have $|P'/P| = |N'/N|\cdot |L'/L|$.
The corresponding $\Q/\Z$-valued quadratic forms satisfy
\[
Q(g(x)) = Q(\bar\pi_P^{-1}(x))-Q(x).
\]
\end{proposition}

\begin{remark} \label{rem:quad}
Let $d:=|L'/L|$.
Using the notation of \eqref{eq:defLS}, the kernel of the homomorphism  $g$ in Proposition \ref{prop:fqm2} is given by $(P'/P)_d$. Its restriction to the orthogonal complement defines an isomorphism
$g^d:(P'/P)^d\to N'/N$, satisfying $Q(g^d(x)) = -Q(x)\in \Q/\Z$.
\end{remark}

\section{The Borcherds lift and the small CM value formula} \label{sect:Review}

\subsection{Borcherds products on Siegel modular threefolds}  \label{sect:Borcherds}


Here we recall the construction of Borcherds products on Siegel modular threefolds in a setup which is convenient for our purposes, see \cite{GN}, \cite{Bo98}.
Recall the definition of the quadratic space $V$ of signature $(3,2)$ and the even lattice $L\subset V$ from Section \ref{sect:siegel3fold}.

In our setting the Borcherds lift is a map from weakly holomorphic vector-valued elliptic modular forms of weight $-1/2$ to meromorphic modular forms on the Siegel threefold with zeros and poles on special divisors.
In \cite{GN} it is described as a map from weakly holomorphic Jacobi forms of weight $0$ to Siegel modular forms. Here we take advantage of the fact that we are working with full level and describe the input space in terms of scalar valued modular forms. We will also obtain more precise information on the multiplier systems of the corresponding Borcherds products.

The discriminant group $L'/L$ of $L$ is isomorphic to $\tfrac{1}{2}\Z/\Z$. We denote by $\frake_0$, $\frake_1$ the standard basis of the group ring $\C[L'/L]$ corresponding to the trivial and the non-trivial coset in $L'/L$. Let $\rho_L$ be the Weil representation of $\Mp_2(\Z)$ on $\C[L'/L]$ as in \cite{Bo98}.
Let $k$ be an even integer. We denote by $M_{k-1/2,\bar\rho_L}^!$ the space of $\C[L'/L]$-valued weakly holomorphic modular forms for $\Mp_2(\Z)$ of weight $k-1/2$ with representation $\bar\rho_L$. Note that by \cite[Theorem 5.1]{EZ} this space is isomorphic to the space of weakly holomorphic Jacobi forms of weight $k$ and index $1$. If $f\in M_{k-1/2,\bar\rho_L}^!$ we write $f=f_0\frake_0+f_1\frake_1$ for the components of $f$ with respect to the basis $\frake_0,\frake_1$.
Then $f_0,f_1$ are scalar modular forms of weight $k-1/2$ for the principal congruence subgroup of level $4$, and $f$ has a Fourier expansion of the form
\[
f(\tau)=\sum_{\mu\in L'/L} \sum_{\substack{m\in \Z-Q(\mu)\\ m\gg -\infty}}c(m,\mu)q^m\frake_\mu .
\]


Let $M_{k-1/2}^{!}$ be the space of scalar-valued weakly holomorphic modular forms of weight $k-1/2$ for the group $\Gamma_0(4)$, and denote its Kohnen plus space by $M_{k-1/2}^{!,+}$.
Any $g\in M_{k-1/2}^{!,+}$ has a Fourier expansion of the form
\begin{align}
\label{eq:fourierg}
g=\sum_{\substack{ N\in \Z\\
N\equiv 0,3\pmod{4}\\N\gg-\infty}} b(N) q^N.
\end{align}

\begin{proposition}
\label{prop:scalvec}
Let the notation be as above. The assignment
\[
f(\tau) =f_0(\tau)\frake_0+f_1(\tau)\frake_1\mapsto \phi(f)=f_0(4\tau) + f_1(4\tau)
\]
defines an isomorphism $\phi:M_{k-1/2,\bar\rho_L}^!\to M_{k-1/2}^{!,+}$.
The $N$-th Fourier coefficient of $\phi(f)$ is given by
\[
b(N)= \begin{cases}c(N/4,0),&\text{if $N\equiv 0\pmod{4}$,}\\
c(N/4,\tfrac{1}{2}),&\text{if $N\equiv 3\pmod{4}$,}\\
0,&\text{if $N\equiv 1,2\pmod{4}$.}
\end{cases}
\]
The inverse map takes $g\in M_{k-1/2}^{!,+}$ with Fourier expansion as in \eqref{eq:fourierg} to
$\vec g = g_0  \frake_0  + g_1 \frake_1$ with
$$
g_\epsilon(\tau) =  \sum_{\substack{N\in \Z\\N \equiv -\epsilon \!\pmod{4}}} b(N) q^{\frac{N}{4}}, \quad   \epsilon =0, 1.
$$
\end{proposition}

\begin{proof}
The transformation law of $f$ under the Weil representation implies that $\phi(f)$ belongs to  $M_{k-1/2}^{!}$. The plus space condition is a consequence of the shape of the Fourier expansion of $f$. This shows that $\phi$ is a well defined injective map. The surjectivity follows for instance from Kohnen's characterization of the plus space in terms of the action of the $U_4$-operator. Alternatively, one can combine Theorems 5.1 and 5.4 of \cite{EZ}.
\end{proof}

We are mainly interested in the case $k=0$. Therefore we now give an explicit description of the space $M_{-1/2}^{!,+}$. Recall that the Cohen Eisenstein series $E_{k-1/2}^+$ of weight $k$ is the unique Eisenstein series in $M_{k-1/2}^+$ with constant term $1$. Its Fourier expansion is given by
\[
E_{k-1/2}^+(\tau)= 1+\frac{1}{\zeta(3-2k)}\!\sum_{\substack{N\in \Z_{>0}\\ N\equiv 0,3\pmod{4}}}\!
\Big[L(2-k,\chi_{-N_0})\sum_{d\mid f}\mu(d)\chi_{-N_0}(d) d^{k-2} \sigma_{2k-3}(f/d)\Big] q^N,
\]
where we have written $N = N_0 f^2$ with $-N_0$  the discriminant of the imaginary quadratic field $\Q(\sqrt{-N})$ and $f\in \Z_{>0}$. Moreover, $\chi_{-N_0}=\leg{-N_0}{\cdot}$ denotes the corresponding quadratic Dirichlet character of conductor $N_0$. Finally, $\sigma_r(n)$ is the sum of $r$-th power of positive divisors of $n$, and $\mu(d)$ the M\"obius function.
The $q$-expansions of these Eisenstein series for $k=4$ and $k=6$ begin with
\begin{align*}
E_{7/2}^+ &= 1 + 56q^3 + 126q^4 + 576q^7 + 756q^8 + 1512q^{11} + 2072q^{12} +\ldots,\\
E_{11/2}^+ &= 1 -88q^3 - 330q^4 - 4224q^7 - 7524q^8 - 30600q^{11} - 46552q^{12} - \ldots .
\end{align*}
Let $E_k= 1-\frac{2k}{B_k}\sum_{n\geq 1} \sigma_{k-1}(n)q^n$ be the classical Eisenstein series of weight $k$ for $\SL_2(\Z)$, and write $\Delta=\frac{1}{1728}(E_4^3-E_6^2)$ for the normalized cusp form of weight $12$. We define weakly holomorphic modular forms in $M_{-1/2}^{!,+}$ by
\begin{align*}
f_1 &= \frac{1}{144\Delta(4\tau)}\left( E_4^2(4\tau) E_{7/2}^+ - E_6(4\tau)E_{11/2}^+\right),\\
f_4 &= \frac{1}{18\Delta(4\tau)} \left( 11E_4^2(4\tau) E_{7/2}^+ +7 E_6(4\tau)E_{11/2}^+\right) .
\end{align*}
The Fourier expansions of these forms begin as follows:
\begin{align}
f_1 &= q^{-1} + 10 - 64q^3 + 108q^4 - 513q^7 + 808q^8+\ldots, \label{eq:f1}\\
f_4 &=q^{-4} + 70 + 32384q^3 + 131976q^4 + 4451328q^7  + \ldots. \label{eq:f4}
\end{align}

Let $M^!_{0,\Z}=\Z[j]$ be the ring of weakly holomorphic modular forms of weight $0$ for $\SL_2(\Z)$ with integral Fourier coefficients, where $j$ denotes the usual Klein $j$-function. Denote by $M_{-1/2,\Z}^{!,+}$ the $\Z$-submodule of forms in $M_{-1/2}^{!,+}$ with integral Fourier coefficients.
Then $M_{-1/2,\Z}^{!,+}$ is a $M^!_{0,\Z}$-module via the map
\[
M^!_{0,\Z}\times M_{-1/2,\Z}^{!,+} \to M_{-1/2,\Z}^{!,+},\quad (h(\tau),f(\tau))\mapsto h(4\tau)f(\tau).
\]

\begin{proposition}
\label{prop:4.2}
The module $M_{-1/2,\Z}^{!,+}$ is a free  $M^!_{0,\Z}$-module of rank $2$. A basis is given by $f_1,f_4$. The constant term in the $q$-expansion of any element of $M_{-1/2,\Z}^{!,+}$ is even.
\end{proposition}

\begin{proof}
By induction on the pole order, it is easily seen that $f_1,f_4$ generate $M_{-1/2,\Z}^{!,+}$.
Considering the leading terms of the $q$-expansions, one finds that the two forms are linearly independent over $M^!_{0,\Z}$. The assertion on the constant term follows by considering the residue pairing of any $f\in M_{-1/2,\Z}^{!,+}$ with the normalized Cohen Eisenstein series of weight $5/2$. The $q$-expansion of the latter series has constant term $1$, while all other coefficients are even.
\end{proof}

\begin{corollary}
\label{cor:4.3}
For every $d\in \Z_{>0}$ with $d\equiv 0,1\pmod{4}$ there exists a unique element $f_d\in M_{-1/2,\Z}^{!,+}$ whose Fourier expansion begins with $f_d= q^{-d}+O(1)$.
\end{corollary}

\begin{remark}
\label{rem:4.4}
Statements analogous to Proposition \ref{prop:4.2} and Corollary \ref{cor:4.3} also hold for the space $M_{-5/2}^{!,+}$. They fail to hold for  $M_{-1/2-2j}^{!,+}$ with $j\in \Z_{\geq 2}$.
\end{remark}

We now turn to Borcherds' theorem \cite[Theorem 13.3]{Bo98}  in our particular case. Recall that for every $\mu\in L'/L$ and every positive $m\in \Z+Q(\mu)$  there is a special divisor $Z(m,\mu)$ on $X=\Sp_4(\Z)\bs \H_2$.
For $d\in \Z_{>0}$ we put
\begin{equation} \label{eq:divisor}
Z(d)=\begin{cases}Z(d/4,0),&\text{if $d\equiv 0\pmod{4}$,}\\
Z(d/4,\tfrac{1}{2}),&\text{if $d\equiv 1\pmod{4}$,}\\
0,&\text{if $d\equiv 2,3\pmod{4}$.}
\end{cases}
\end{equation}
It is easily seen that $Z(d)$ is the Humbert surface of discriminant $d$ on $X$.
Any $f\in M_{-1/2,\Z}^{!,+}$ with Fourier coefficients $c(N)$ determines a special divisor
\[
Z(f) = \sum_{d\in \Z_{>0}} c(-d) Z(d).
\]

Also recall from \cite{Ma}, that the group $\Sp_4(\Z)$ only admits multiplier systems of integral weight, which are then unitary characters $ \Sp_4(\Z)\to \C^\times$. The group of such characters has order two and is generated by the character of the Siegel modular form $\Delta^{(2)}$ of weight $5$ given by the product of the 10 even theta constants.

\begin{theorem}[Borcherds] \label{thm:Borcherds}
Let  $f\in M_{-1/2,\Z}^{!,+}$ with Fourier coefficients $c(N)$. There exists a meromorphic Siegel modular form $\Psi(f)$ for  $\Sp_4(\Z)$ (with a quadratic character)  such that:
\begin{itemize}
\item[(i)] The weight of $\Psi(f)$ is $c(0)/2$.
\item[(ii)] The divisor of $\Psi(f)$ is $Z(f)$.
\item[(iii)] The function
$$
\Phi(z,f) = -c(0)( \log(4\pi) + \Gamma'(1))
             - 2 \log ( |\Psi(z,f)|^2 \det(y)^{c(0)/2} )
$$
is the regularized theta integral of $\vec f$ as defined in \cite{Bo98}, see also  \cite[(4.7)]{BY09}.
\item[(iv)] Let $W\subset \Sym_2(\R)_{>0}$ be a Weyl chamber for $f$, and let $\rho_{W,f}\in \Sym_2(\Q)$ be the Weyl vector associated with $W$ and $f$ as in \cite{Bo98}. The function $\Psi(f)$ has the infinite product expansion
\[
\Psi(z, f)= e(\tr(\rho_{W,f}z))\prod_{\substack{T\in \Sym_2(\Z)^\vee\\ (T,W)>0}} \big( 1-e(\tr(Tz))\big)^{c(4\det(T))},
\]
which converges if  $\det(y)$ is sufficiently large. Here $\Sym_2(\Z)^\vee$ denotes the lattice of half-integral symmetric $(2\times 2)$-matrices, and $z=x+iy\in \H_2$.
\end{itemize}
\end{theorem}

If $m$ is a positive integer with $m\equiv 0,1\pmod{4}$ we write $\Psi_m:=\Psi(f_m)$ for the Borcherds lift of the element $f_m\in M_{-1/2,\Z}^{!,+}$ of Corollary~\ref{cor:4.3}. It has integral weight and divisor $Z(m)$.
This divisor and (the square of) the Borcherds product $\Psi_m$ can be extended to the moduli stack $\calA_2$ over $\Z$ as follows.
Let $S$ be a scheme and let $\mathbf A \in \calA_2(S)$ be an $S$-valued point. Following  \cite[Section 2]{KRSiegel} we call an endomorphism $\phi\in \End_S(\mathbf A)$ \emph{special}, if it is symmetric (i.e. invariant under the Rosati involution) and of trace $0$. If $\phi$ is a  special endomorphism, then $\phi^2\in \Q\id_{\mathbf A}$. The quadratic form $\phi\mapsto \phi^2$ on the lattice of special endomorphisms is positive definite.

Let $\mathcal Z(m)$ be the moduli stack over $\Z$ whose functor of points is given by assigning to a scheme $S$ a pair $(\mathbf A, j)$  where $\mathbf A \in \mathcal A_2(S)$ and $j\in\End_S(\mathbf A)$ is a special endomorphism with $j^2 = m  \id_\mathbf A$.
Let $\mathcal M_k$ be the line bundle over $\mathcal A_2$ of Siegel modular forms of weight $k$.

\begin{proposition} \label{prop:integral}
Let the notation be as above.

(1)  The forgetful map  $\mathcal Z(m) \rightarrow \mathcal A_2$ is finite and unramified. It defines a horizontal divisor on $\mathcal A_2$ with complex points $\calZ(m)(\C)=Z(m)$.

(2)  The Borcherds product $\Psi_m^2$ extends to a section of $\mathcal M_{2k}$ on $\mathcal A_2$ with divisor $\dv \Psi_m^2 = 2\mathcal Z(m)$. Here $k= c(0)/2$ is determined by the constant term $c(0)$ of $f_m$.

\end{proposition}

\begin{proof}
(1) By \cite[Proposition 2.5]{KRSiegel}, the forgetful map $\mathcal Z(m)\to \calA_2$ is finite, unramified, and determines a divisor on $\mathcal A_2$ with complex points $Z(m)$.
%
The following argument, which was communicated to us by Ben Howard, shows that $\mathcal Z(m)$ is  horizontal.

Let $p$ be a prime.
Since $\mathcal A_2$ is irreducible over $\bar{\mathbb F}_p$ (see  Corollary 1 in Chapter IV.6.8 of \cite{FC}),  it suffices to show that $\mathcal Z(m)(\bar{\mathbb F}_p)$ is not the whole fiber $\mathcal A_2(\bar{\mathbb F}_p)$.
Let $(A, \lambda) $ be a simple abelian surface with  principal polarization $\lambda$  over $\bar{\mathbb F}_p$ for which $K=\End^0(A) $ is  a quartic CM field with real quadratic subfield $F=\Q(\sqrt D)$, where $D$ is a positive fundamental discriminant. Hence the space of special endomorphisms of $(A, \lambda)$  is equal to the line  $\Q \sqrt D$.   Consequently $(A, \lambda) \notin \mathcal Z(m)(\bar{\mathbb F}_p)$ unless $m=D b^2$ (which can happen for at most one $D$). So   given $m$,  there is always some principally polarized abelian surface $(A, \lambda) \notin \mathcal Z(m)(\bar{\mathbb F}_p)$.
Indeed, for a fixed $m =D b^2$, choose a non-biquadratic CM field $K$ which does not contain  $\Q(\sqrt D)$ and such that $p$ splits completely in $K$. If $(A, \lambda)$ is a principally polarized abelian surface over some number field with CM by $K$, then its reduction at some prime over $p$ is ordinary, simple, and has CM by $K$. Thus it does not belong to  $\mathcal Z(m)(\bar{\mathbb F}_p)$.

(2) The product expansion in Theorem \ref{thm:Borcherds} of the Siegel modular form  $\Psi_m^2$   implies that it has integral and coprime Fourier coefficients. Because $\Psi_m^2$ is a square, it has trivial character. By the $q$-expansion principle, this implies that $\Psi_m^2$ extends to a rational section of $\mathcal M_{2k}$ whose divisor is horizontal.  Since the divisors $\dv(\Psi_m^2)$ and $2\mathcal Z(m)$ agree on the generic fiber
and since both are horizontal, they agree on the whole moduli stack $\calA_2$.
\end{proof}

\begin{example}
The theorem implies that $\Psi_1^2$, the Borcherds lift of $2f_1$, is a Siegel modular form of weight $10$ with divisor $2\calZ(1)$. Note that $\calZ(1)$ is the decomposable locus of $\calA_2$, that is, the image of the natural map  $\mathcal A_1 \times \mathcal A_1 \rightarrow \mathcal A_2$. Hence $\Psi_1$ must be a multiple of the Siegel modular form $\Delta^{(2)}$ given by the product of the 10 even theta constants. By looking at the Fourier expansion one finds that
\[
\Psi_1 = 2^{-6} \Delta^{(2)}.
\]
Consequently,  $\Psi_1^2$ is equal to the Igusa cusp form $\chi_{10}$ (normalized to have integral and coprime Fourier coefficients).
\end{example}

\subsection{Incoherent Eisenstein series and the holomorphic part of its derivative}
\label{sect:inceis}

This subsection gives a slight generalization of \cite[Theorem 2.6]{BY09}. For the explicit calculation of local Whittaker functions, see \cite[Section 5]{YYY}.
Let $d<0$ be a fundamental discriminant and let $\kay=\Q(\sqrt d)$ be the corresponding imaginary quadratic field with ring of integers $\co_\kay =\co_d$. Let $\mathfrak a = \Z \omega_2 + \Z a$ be a quadratic lattice in $\kay$  as in \eqref{eq:omega2} with endomorphism ring $\co_{\mathfrak a} =\co_{d_2}$, $d_2 =d t_2^2$. For $ t \in  \Z_{>0}$
we consider the lattice  $\mathcal N = \mathfrak a$ with the quadratic form $Q(x) = - \frac{t}a x \bar x$, and $U=\mathcal N \otimes_\Z \Q =\kay$.
Then
$\mathcal N' =\frac{1}{\sqrt{d_1}} \mathfrak a $, $d_1 =d_2 t^2 =d t_1^2$,  $t_1=t_2 t$. We will need the incoherent Eisenstein series of weight $1$ associated to $\mathcal N$ and the holomorphic part of its central derivative. We now summarize what we need and refer to \cite[Section~5]{YYY} and \cite[Section~5]{LYY} for details (see also  \cite{BY09}, \cite{KY10}, \cite{BKY}).  We begin by setting up some notation.

Let $\chi=\chi_{\kay/\Q} =\prod \chi_p$ be the quadratic idele class character of $\A^\times$ associated to the quadratic field extension $\kay/\Q$, and  denote by $I(s,\chi)=\otimes' I(s, \chi_p) $ be the associated  induced principal series representation of $\SL_2(\A)$. Let $\mathcal  C =\otimes' \mathcal C_p$ be the incoherent binary quadratic space over $\A$ such that $\mathcal C_p = U_p$ for $p <\infty$ and $\mathcal C_\infty$ is positive definite of dimension $2$. Denote by
$$
\lambda: S(\mathcal C_p ) \rightarrow I(0, \chi_p), \quad \lambda(\phi)(g) =\omega(g)\phi(0),
$$
the (equivariant) Rallis map, where $\omega$ is the Weil representation of $\SL_2(\Q_p)$ (respectively $\SL_2(\A)$) on $S(\mathcal C_p)$ (respectively  $S(\mathcal C)$) with respect to the standard additive character $\psi=\prod \psi_p$. We also write $\lambda(\phi)$  for the associated standard section of $I(s, \chi)$ with $\lambda(\phi)(g, 0) = \lambda(\phi)(g)$. Let $\phi_\infty (x) =e^{-2 \pi Q(x)}$, then $\lambda(\phi_\infty)$ is the standard weight $1$ spherical section $\Phi_\infty^1$.  For any $\phi \in S(U(\A_f)) =S(\mathcal C_f)$, we have an incoherent Eisenstein series
\begin{equation}
E(\tau, s, \phi) = v^{-1/2} E(g_\tau, s, \phi) = v^{-1/2} \sum_{\gamma \in \Gamma_\infty \backslash \SL_2(\Z)} \lambda(\phi)(\gamma_f, s) \lambda(\phi_\infty)(\gamma_\infty g_\tau, s)
\end{equation}
of weight $1$, where $\Gamma_\infty$ denotes the stabilizer of $\infty $ in $\SL_2(\Z)$. This Eisenstein series has an odd functional equation under $s\mapsto -s$ and hence satisfies $E(\tau, 0, \phi)=0$.
The central derivative  $E'(\tau, 0, \phi)$ is a harmonic Maass form of weight $1$. Let
\begin{equation}
\mathcal E(\tau, \phi) = \sum_{m \ge 0} \kappa (m, \phi) q^m
\end{equation}
denote the holomorphic part of $E'(\tau, 0, \phi)$  with Fourier coefficients
\begin{equation} \label{eq:EisCoeff}
\kappa(m, \phi)q^m = \begin{cases}
E_m'(\tau, 0, \phi) &\hbox{if }  m >0,
\\
E_m'(\tau, 0, \phi) - \phi(0) \log v  &\hbox{if }  m =0.
\end{cases}
\end{equation}
The difference
$$
E'(\tau, 0, \phi) -\mathcal E(\tau, \phi) - \phi(0) \log v
$$
decays rapidly when $v \rightarrow \infty$. We denote by  $\phi_\mu$ the characteristic function of the coset $\mu+ \hat{\mathcal N}\subset U(\A_f)$. We put $\kappa_\mathcal N(m, \mu) = \kappa(m, \phi_\mu)$ and
\begin{equation} \label{eq:Eisenstein}
\mathcal E_\calN(\tau) = \sum_{\mu \in \mathcal N'/\mathcal N}  \mathcal E(\tau, \phi_\mu)  \phi_\mu =\sum_{\substack{m \ge 0 \\\mu \in \mathcal N'/\mathcal N}} \kappa_\calN(m, \mu) q^m \phi_\mu .
\end{equation}

\begin{theorem}
\label{theo4.6}
Let the notation be  as above.
\label{thm:eiscoeff}
\begin{enumerate}
\item[(1)]
If  $m>0$, then $\kappa_{\mathcal N}(m, \mu) =\sum_p a_p \log p$ where the sum runs over all primes $p$, and $a_p\in \Q$.
Moreover, if $a_p\ne 0$ then the following holds:
\begin{enumerate}
\item[(a)] $p$ is inert or ramified in $\kay$;

\item[(b)]  if  $p\nmid d_1$, then $p$ is inert and  $o_p(m) \ge 1$ is odd. Here $o_p(m)$ is the largest integer $r$ such that $mp^{-r}$ is $p$-integral.
\end{enumerate}

\item[(2)] We have
$$
\kappa_{\mathcal N}(0, \mu) = -2 \frac{\Lambda'(0, \chi)}{\Lambda(0, \chi)} \delta_{0, \mu} -\prod_{p|d,\,p\nmid t_1} \delta_{0, \mu_p}  \cdot \frac{d}{ds}\Big|_{s=0}\prod_{p|t_1} \tilde{W}_{0, p}(1, s, \phi_\mu).
$$
Here
$\mu_p$ is the $p$-th component of $\mu$,
$$
\tilde W_{0, p} (1, s, \phi_\mu) = \frac{L_p(s+1, \chi)}{L_p(s, \chi)} \cdot \frac{W_{0, p}(1, s, \phi_\mu)}{\gamma(\mathcal C_p) |d|_p^{\frac{1}2}}
$$
is a renormalized local Whittaker function, and
$$
\Lambda(s, \chi) = |d|^{\frac{s}2} L(s, \chi) \pi^{-\frac{s+1}2}\Gamma(\frac{s+1}2) =|d|^{\frac{s}2} \prod_{p\le \infty} L_p(s, \chi)
$$
is  the complete $L$-function of $\chi$ with functional equation $\Lambda(1-s, \chi) =\Lambda(s, \chi)$. Notice that $\Lambda(0, \chi) = \frac{2 h_d}{w_d}$, where $h_d$ denotes the  class number of $\kay$ and $w_d$  the number of units in $\calO_\kay$.
\end{enumerate}
\end{theorem}

\begin{proof}
The claims follow from \cite[Section 2]{BY09}, \cite[Section 5]{YYY}, and \cite[Section 4]{HY12}.
We briefly sketch the basic idea and leave the details to the reader. For $m >0$,  write the $m$-th Fourier coefficient
$$
E_m(\tau, s, \phi_\mu)   =W_{m, \infty}(\tau, s)  \prod_{p <\infty} W_{m, p}(1, s, \phi_\mu)
$$
as a product of local Whittaker functions,  and $W_{m, \infty}(\tau, 0) =c_\infty q^m $ for some $c_\infty \ne 0$.  Notice that the local Whittaker function
$$
W_{m, p}(1, s, \phi_\mu) =\int_{\Q_p} \lambda(\phi_\mu)(w n(b)) (0) | a(w n(b))|_p^{s} \psi_p(- bm) db
$$
is a polynomial of $p^{-s}$.  It is zero unless $m \in  \frac{t}{a}
\frac{\norm(\mathfrak a)}{d_1} \Z_p= \frac{1}{d_2 t} \Z_p$ by
\cite[Appendix A]{YYY}.  So $E_m(\tau, s, \phi_m)=0$ unless $m \in
\frac{1}{d_2 t}\Z$.

When $U_p$ does not represent $m$, we have $W_{m, p}(1, 0, \phi_\mu)=0$ and
$W_{m, p}'(1, 0, \phi_\mu)  = a_p \log p$ for some rational number $a_p$.
Since $\mathcal C$ is incoherent, there is always some $p<\infty$ which is non-split in $\kay$  such that  $\mathcal C_p=U_p$  does not represent $m$.  In such a case,
$
E_m(\tau, 0, \phi_\mu)=0
$, and
$$
E_m'(\tau,0, \phi_\mu) =\left[ c_p c_\infty \prod_{q \ne p, \infty} W_{m, q}(1, s, \phi_\mu)\right]_{s=0} \log p= a_p  \log p
$$
for some rational number $c_p$.
Next, when  $p \nmid d_1$, and $a_p\ne 0$, then  $p$ is inert and $\phi_\mu =\phi_0$. In such a case, $o_p(m) \ge 1$ is odd.

Assertion (2) follows from
\begin{equation}
 E_0(\tau, s, \phi_\mu) =  v^{-\frac{s}2} \phi_\mu(0) + W_{0, \infty}(\tau,s, \phi_\mu) \prod_{p<\infty} W_{0, p}(1, s, \phi_\mu).
\end{equation}
A local calculation shows
$$
W_{0, \infty}(\tau,s, \phi_\mu) \prod_{p<\infty} W_{0, p}(1, s, \phi_\mu) =- v^{\frac{s}2} \frac{\Lambda(s, \chi)}{\Lambda(-s, \chi)} \prod_{p|d_1} \tilde{W}_{0, p}(1, s, \phi_\mu)
$$
for the renormalized local Whittaker function
$
\tilde W_{0, p} (1, s, \phi_\mu)
$,
 which is a rational function of $p^{-s}$.
  The fact that $E_0(\tau, 0, \phi_\mu)=0$ implies that
\begin{equation} \label{eq:value0}
\prod_{p|d_1} \tilde{W}_{0, p}(1, 0, \phi_\mu) =\phi_\mu(0),
\end{equation}
and
\begin{equation} \label{eq:der0}
E_0'(\tau, 0, \phi_\mu) =-2 \frac{\Lambda'(0, \chi)}{\Lambda(0, \chi)} \phi_{\mu}(0)   - \phi_\mu(0) \log v  - \frac{d}{ds}\Big|_{s=0}\prod_{p|d_1} \tilde{W}_{0, p}(1, s, \phi_\mu).
\end{equation}
Finally,  when  $p |d$ but $p \nmid t_1$, \cite[Corollary 5.6, Proposition 5.8(3)]{YYY} implies
\begin{equation} \label{eq4.12}
\tilde W_{0, p}(1, s, \phi_\mu) = \delta_{0, \mu_p}.
\end{equation}
This proves (2).
\end{proof}

We remark that the local Whittaker functions are computed in \cite[Section 5]{YYY} so that the coefficients $\kappa_\mathcal N(m,\mu)$ are computable. In fact, when $t=1$ and $d_2=d$, an explicit formula was given in \cite[Theorem 2.6]{BY09}, which we recall here for convenience:
\begin{align} \label{eq:kappa}
-\Lambda(0, \chi) \kappa_{\mathcal N}(m, \mu)
 &=  \eta_0(m, \mu)  \sum_{\text{$p$ inert} } (\ord_p(m) +1) \rho(m|d|/p) \log p  \notag
  \\
   &\qquad +  \rho(m|d|) \sum_{p|d} \eta_p(m, \mu)
(\ord_p(m) +1) \log p,
  \end{align}
  where
\begin{align*}
\eta_p(m, \mu) &=(1-\chi_p(-m\norm(\mathfrak a))) \prod_{\substack{q|d, \, q\ne p\\
\mu_q =0}} (1+ \chi_q(-m\norm(\mathfrak a))),\\
\eta_0(m, \mu) &= \prod_{\substack{q|d\\
\mu_q =0}} (1+ \chi_q(-m\norm(\mathfrak a))).
\end{align*}
Here  we take $\eta_0(m, \mu)=1$ and $\eta_p(m, \mu)=0$ if $\mu_q\ne
0$ for all $q|d$. Finally,
$
\rho(n)
$
is the number of integral ideals of $\kay$ with norm $n$, and
\begin{equation} \label{eq:kappa0}
\kappa_\mathcal N(0, \mu) =-2 \frac{\Lambda'(0, \chi)}{\Lambda(0, \chi)} \delta_{0, \mu}.
\end{equation}

\subsection{The small CM value formula}
\label{sect:4.3}

We recall Schofer's small CM value formula in our special case (\cite{Schofer}, see also \cite[Theorem 1.3]{BY09} for a generalization). Let $ L, \mathcal P$ and $\mathcal N$ be as in Section \ref{sect:set-up}. Then $\calN_\R$ gives rise to  a small CM cycle $Z(\calN)$ on $X_L$, see \eqref{eq:smallcm}.
Let
$$
\mathcal E_{\mathcal N}(\tau)=\sum_{\substack{m \ge 0\\\mu \in \mathcal N'/\mathcal N}} \kappa_{\mathcal N}(m, \mu) q^m \phi_\mu
$$
be the holomorphic part of the  central derivative  of the incoherent Eisenstein series defined in (\ref{eq:Eisenstein}), and let
$$
\theta_\mathcal P(\tau) = \sum_{\substack{m \ge 0\\ \mu \in \mathcal P'/\mathcal P}} a_\mathcal P(m, \mu) q^m \phi_\mu
$$
be the vector valued theta series of weight $3/2$ associated to $\mathcal P$, where $a_\mathcal P(m, \mu)$ is the number of representations of $\mu + \mathcal P$ representing $m$. The following theorem is a consequence of the main result of Schofer's thesis \cite{Schofer}.

\begin{theorem}[Schofer]
\label{thm:SmallCM}
Let $f =\sum c_f(m) q^m  \in M_{-\frac{1}2}^{!, +}$.
Then
$$
\Phi(Z(\mathcal N), f) = \deg Z(\mathcal N) \cdot\operatorname{CT}\left[\langle\vec f ,\mathcal E_{\mathcal N} \otimes \theta_\calP\rangle\right],
$$
where $\operatorname{CT}[\cdot]$ denotes the constant term of a $q$-series and $\langle\cdot,\cdot \rangle$ the natural $\C$-bilinear pairing on the space of the Weil representation as in \cite[Section 3]{BY09}. Equivalently, we have
$$
\Phi(Z(\mathcal N), f) = \deg Z(\mathcal N) \sum_{\substack{m\in \Z_{\geq 0}}}
c_f(-m)
\sum_{\substack{\mu_1 \in \mathcal N'/\mathcal N\\ \mu_2 \in \mathcal P'/\mathcal P \\
\mu_1+\mu_2\equiv \frake_{m}\,(L)}}\sum_{\substack{m_1,m_2\in \Q_{\ge 0}\\ m_1 + m_2 = \frac{m}4}}
	\kappa_{\mathcal N}(m_1, \mu_1) a_\mathcal P(m_2, \mu_2),
$$
where $\frake_m\in L'/L$ denotes the neutral (respectively non-neutral) element if $m$ is even (respectively odd). The second sum on the right hand side runs through all $\mu_1 \in \mathcal N'/\mathcal N$ and $\mu_2 \in \mathcal P'/\mathcal P$ for which the image of $\mu_1+\mu_2$ under the natural map to $(\calN'\oplus \calP')/L$ lies in the subgroup $L'/L$ and is equal to $\frake_m$.
\end{theorem}

\section{The Faltings height and the main theorems} \label{sect:MainFormula}

\subsection{The Faltings height -- set up and an example}
\label{sect:fh}

Let $A$ be an abelian variety over a number field $K$ with semi-stable reduction everywhere, and let $\mathcal A$ be its Neron model over $\co_K$ with zero section $s: \co_K  \rightarrow \mathcal A$.  Let $\omega_{\mathcal A} =s^* \Omega_{\mathcal A/\co_K}^{\dim(A)}$ be the pull back of the canonical line bundle on $\co_K$ with the following normalized natural metric: for each embedding $\sigma: K  \rightarrow \C$, and a section $\alpha$ of $\omega_\mathcal A$,
$$
\| \alpha \|_\sigma^2 = \frac{1}{(2 \pi)^g} \int_{A_\sigma(\C)} \alpha \wedge \bar\alpha.
$$
This determines a metrized line bundle $\widehat{\omega}_\mathcal A = ({\omega}_\mathcal  A , \{\| \, \|_\sigma,  \sigma: K \rightarrow \C\})$. Then the (stable) Faltings height of $A$ is defined  as
\begin{equation}
h_{\Fal}(A)  = \frac{1}{[K:\Q]} \widehat{\deg} (\widehat{\omega}_\mathcal A)
             = \frac{1}{[K:\Q]} \left[ \log| \omega_{\mathcal A}/\co_K \alpha| - \sum_{\sigma} \log\| \alpha \|_\sigma \right]
\end{equation}
for any non-zero section $\alpha$ of $\omega_\mathcal A$.  It does not depend on the choice of $\alpha$.

If $A$ is endowed with a principal polarization $\lambda$ which is also defined over $K$, then  $(A, \lambda)$ defines a point in $\mathcal A_g(K)$, where $\mathcal A_g$ is the moduli stack over $\Z$ of principally polarized abelian schemes of relative dimension $g$. In this case, the Faltings height can be calculated via an arithmetic intersection (see \cite{Ya10}). Let $\mathcal M_k$ be the line bundle  over $\mathcal A_g$ of Siegel modular forms of weight $k$ and genus $g$, and let
$$
\| f(\tau) \|_{\Pet} =(4 \pi)^{gk/2}  |f(\tau) (\det \operatorname{Im} \tau)^{\frac{k}2 }|
$$
be the normalized  Petersson metric on  $\mathcal M_k(\C)$. In this way, we obtain a metrized line bundle $\widehat{\mathcal M}_k= (\mathcal M_k,  \| \cdot \|)$ over $\mathcal A_g$, and there is a natural  isomorphism  $\widehat{\mathcal M}_k\cong\widehat{\omega}^{\otimes k} $, where $\widehat{\omega}$ denotes the metrized Hodge bundle on $\calA_g$. For an  element $\mathbf A= (A, \lambda) \in \mathcal A_g(K)$,  let $\underline{\mathbf A} $ be its N\'eron model over $\co_K$.  We have by \cite[Section 2]{Ya10}
\begin{equation} \label{eq:FaltingsHeight}
k \cdot h_{\Fal}(A) = \frac{1}{ [K:\Q]} \big(\underline{\mathbf A}\cdot \dv(\Psi) - \sum_{\sigma: K \hookrightarrow \C} \log \| \Psi((A^\sigma, \lambda^\sigma))\|_{\Pet} \big)
\end{equation}
for any rational section  $\Psi$ of  $\mathcal M_k$ over $K$, where
\begin{equation}
 \underline{\mathbf A} \cdot \dv(\Psi) =\sum_{\mathfrak p} e_{\mathfrak p} \log \norm(\mathfrak p).
\end{equation}
Here $e_{\mathfrak p}$ is the local intersection index, that is, the greatest integer $n$ such that $\underline{\mathbf A} \in  \dv(\Psi) \mod \mathfrak p^n$, and the value $\Psi(\mathbf A^\sigma)$ is the usual value (not in the stack sense).  Indeed, by \cite[Section 2]{Ya10}, we have in our current notation
$$
\frac{k \cdot h_{\Fal}(A)} { |\Aut(\mathbf A)| } =h_{\widehat{\mathcal M}_k}(\mathbf A)= \frac{1}{ [K:\Q]} \bigg(\frac{(\mathcal A, \lambda)\cdot \dv(\Psi)}{|\Aut(\mathbf A)|} - \sum_{\sigma: K \hookrightarrow \C}\frac{ \log \| \Psi(\mathbf A^\sigma)\|_{\Pet}}{|\Aut(\mathbf A^\sigma)|} \bigg).
$$
Since $|\Aut(\mathbf A^\sigma)|= |\Aut(\mathbf A)|$, we obtain (\ref{eq:FaltingsHeight}) by  multiplying  $|\Aut(\mathbf A)|$ to both sides of the above identity.
We remark that $h_{\Fal}(A)$ does not depend on the choice of $K$ or on the choice of the polarization $\lambda$.
Finally, recall that $h_{\Fal}(A_1\times A_2) = h_{\Fal}(A_1) + h_{\Fal}(A_2)$.

We end this subsection with a proposition which is needed in the proof of our main formula in next subsection. The proof is actually similar to that of the main formula.

\begin{proposition} \label{prop5.3}  Let $E_1 =\C/\mathcal O_{d_1}$ and $E_2=\C/\mathfrak a$ be two CM elliptic curves where the order of $\mathfrak a$ is $\mathcal O_{d_2}$ with  $d_1 =d_2 t^2 =d t_1^2$ as before. Let $\mathcal N =(\bar{\mathfrak a},  -\frac{t}a \norm)$ be as in  Proposition \ref{prop:N}. Then
$$
h_{\Fal}(E_1) + h_{\Fal} (E_2) = \frac{1}2 \kappa_{\mathcal N}(0, 0) + \frac{1}2 (\log 4\pi + \Gamma'(1)).
$$
\end{proposition}
\begin{proof} Let $Y_0$ be the Shimura variety associated to the quadratic lattice $(M_2(\Z), \det )$, which is well-known to be isomorphic to the product $Y_0(1) \times Y_0(1)$ of classical modular curves, whose canonical model $\mathcal Y_0$ over $\Z$  is isomorphic to $\mathcal A_1 \times \mathcal A_1$.  Let $\mathcal X_0= \mathcal X_0(1) \times \mathcal X_0(1)$ be its compactification.  Then  $\mathbf A_0= (A(\mathfrak a, t), \lambda_0=\lambda_1 \times \lambda_2)$, the product $E_1 \times E_2$ with the product polarization, belongs to $\mathcal Y_0(H_{d_1})$.  Choose a finite field extension  $K$ of $H_{d_1}$ such that $\mathbf A_0$ has good reduction everywhere, i.e.,  its Neron model $\underline{\mathbf A_0}$ belongs to $\mathcal Y_0(\mathcal O_K)$. On the other hand,   $\Delta(z_1) \Delta(z_2)  =\Psi(z_1, z_2, 24)$ is the Borcherds lifting of the constant function $24$  with divisor being the boundary $\mathcal X_0^b= \mathbb P^1 \times \{ \infty\} \cup \{ \infty \} \times \mathbb P^1 $ of $\mathcal X_0$, and
$$
- 4 \log \| \Psi(z, 24)\|_{\Pet}  -24 C_0 = \Phi(z, 24)
$$
is the regularized theta  integral of $24$ as defined in \cite{Bo98}, see also  \cite[(4.7)]{BY09}, where $z=(z_1, z_2) \in \H^2$ and $C_0 = \log 4\pi + \Gamma'(1)$. By (\ref{eq:FaltingsHeight}),
\begin{align*}
&12(h_{\Fal}(E_1) + h_{\Fal} (E_2)) = h_{\widehat{\mathcal M}_{12}}(\underline{\mathbf A_0})
\\
&= \frac{ \underline{\mathbf A_0} \cdot \mathcal X_0^b}{[K:\Q]} - \frac{1}{[K:\Q]}\sum_{\sigma: K \hookrightarrow \C}  \log \| \Psi(\mathbf A_0^\sigma, 24)\|_{\Pet}
\\
&= 6 C_0 + \frac{1}{8h_{d_1}} \sum_{ \sigma: H_{d_1} \hookrightarrow \C}  \Phi(\mathbf A_0^\sigma, 24 ).
\end{align*}
Here we used the fact that $\underline{\mathbf A_0}$ has good reduction everywhere and does not intersect the boundary $\mathcal X_0^b$.
As in  \cite[Section 2]{LYY}, we have that
$$
\sum_{\sigma: H_{d_1} \hookrightarrow \C} \mathbf A_0^\sigma =Z_0(\mathcal N)
$$
is the small CM cycle $Z_0(\mathcal N)$ associated to $\mathcal N$ in $X_0$ (as an $\Orth(2,2)$ Shimura variety). So the small CM value formula in  Schofer's thesis \cite{Schofer} (see Theorem  \ref{thm:SmallCM} for a similar  $\Orth(3, 2)$-case) gives
\begin{align*}
&\sum_{ \sigma: H_{d_1} \hookrightarrow \C}  \Phi(\mathbf A_0^\sigma, 24 )= \Phi(Z_0(\mathcal N), 24)
=2h_{d_1} \cdot  24 \kappa_{\mathcal N}(0, 0).
\end{align*}
Therefore,
$$
h_{\Fal}(E_1) + h_{\Fal} (E_2))= \frac{1}2 C_0 + \frac{1}2 \kappa_{\mathcal N}(0, 0)
$$
as claimed.
\end{proof}

 \begin{remark}
In Proposition \ref{prop5.3}, take  $E_2 =E_1$, so that we have $t=1$ and $\mathcal N =(\mathcal O_{d_1},  -\norm) $. We obtain
 $$
 2 h_{\Fal}(E_1) = \frac{1}2( C_0 + \kappa_{\mathcal N}(0, 0)).
 $$
 When  $d_1=d$ is fundamental,  then  \eqref{eq:kappa0} implies the well-known Chowla-Selberg formula as reformulated by Colmez:
 $$
 2 h_{\Fal}(E_1) = -\frac{\Lambda'(0,\chi_d)} {\Lambda(0, \chi_d)}  +  \frac{1}2 (\log 4\pi + \Gamma'(1)).
 $$

 In the general case,  we have  by Theorem \ref{theo4.6} and the Chowla-Selberg formula above
 $$
 h_{\Fal} (E_1) = h_{\Fal}(E) - \frac{1}2 \frac{d}{ds}\Big|_{s=0}\prod_{p|t_1} \tilde{W}_{0, p}(1, s, \phi_{0, p} ).
 $$
with $\phi_{0, p} = \hbox{Char}( \mathcal N_p)$.  By (\ref{eq:value0}) and (\ref{eq4.12}), we have
$$
\tilde{W}_{0, p}(1, 0, \phi_{0, p}) =1,
$$
and
$$
\frac{d}{ds}\Big|_{s=0}\prod_{p|t_1} \tilde{W}_{0, p}(1, s, \phi_{0, p}) =\sum_{p|t_1} \frac{\tilde{W}'_{0, p}(1, 0, \phi_{0, p} )}{\tilde{W}_{0, p}(1, 0, \phi_{0, p} )}
$$
For $p\ne 2$, by \cite[Proposition 5.1]{LYY} we have
$$
 C \tilde{W}_{0, p}(1, s, \phi_{0, p})=  p^{o_p(t_1)(1-s)} + \frac{L_p(s+1, \chi_d)}{L_p(s, \chi_d)} (1 -p^{-s}) \sum_{0 \le k < o_p(t_1)} p^{k(1-s)}
$$
for some non-zero constant $C$. Hence
$$
\frac{\tilde{W}'_{0, p}(1, 0, \phi_{0, p} )}{\tilde{W}_{0, p}(1, 0, \phi_{0, p} )} = - \sum_{p|t_1} \left[o_p(t_1) - \frac{1 -\chi(p)}{p-\chi(p)} \frac{ 1- p^{-o_p(t_1)}}{1-p^{-1}} \right] \log p.
$$
 So we obtain the following general Colmez formula for CM elliptic curves
 $$
 h_{\Fal}(E_1) = h_{\Fal}(E) + \frac{1}2 \sum_{p|t_1} \left[o_p(t_1) - \frac{1 -\chi(p)}{p-\chi(p)} \frac{ 1- p^{-o_p(t_1)}}{1-p^{-1}} \right] \log p
 $$
 up to a multiple of $\log 2$. A similar calculation should give the identity for $\log 2$-contribution,  too.  This formula   was already proved in \cite[Theorem 8.5]{Mocz} by Mocz  and in \cite{NT}  by Nakkajima and Taguchi. They both use Faltings'
isogeny theorem, which is quite different from the above proof.
 \end{remark}

\subsection{The main intersection formula}
\label{sect:5.2}
We are now ready to state and prove our main intersection formula.
Let  $\mathbf A =(A(\mathfrak a, t),  \lambda)$  be the product of two CM elliptic curves endowed with a principal polarization $\lambda=\lambda_0 \circ \alpha$ as in Proposition \ref{prop:Dagger}, where $\lambda_0$ is the product polarization and $\alpha =\kzxz {\alpha_1} {\frac{t}a \bar \beta }{ \beta} {\alpha_2}  \in  \tilde{\mathcal P}_0$ such that $\alpha_i >0$ and $\det \alpha =1$.  Recall that
$$
A(\mathfrak a, t) = \C/\mathcal O_{d_1} \times \C/\mathfrak a =E_1 \times E_2,   \quad \mathcal O_{\mathfrak a} =\co_{d_2},\quad   d_i =d t_i^2,  \hbox{ and }  t_1 =t t_2,
$$
and that $\mathbf A$ is defined over $H_{d_1}$. Here $E_1= \C/\mathcal O_{d_1}$ and $ E_2=  \C/\mathfrak a$.

Let $K$ be a finite field extension of $H_{d_1}$ such that the Neron model $\underline{\mathbf A}$ of $\mathbf{A}$ over $\co_K$ is smooth, i.e., has good reduction everywhere, and defines a point  $\underline{\mathbf A} \in  \mathcal A_2(\co_K)$. We are ready to give a formula for the normalized arithmetic intersection number
\begin{equation}
\frac{1}{[K:H_{d_1}]}\underline{\mathbf A} \cdot \mathcal Z(m) = \frac{1}{[K:H_{d_1}]} \sum_{\mathfrak p }  e_{\mathfrak p} \log \norm(\mathfrak p),
\end{equation}
when  the intersection is proper, that is,  $\mathbf A \notin Z(m)$.   Here the sum runs over all prime ideals $\frakp$ of $\co_K$, and $ e_{\mathfrak p}$ is the largest non-negative integer $n$ such that $\underline{\mathbf A}\in \mathcal Z(m) $ modulo $ \mathfrak p^n$.

 Notice that the normalized intersection number is independent of the choice of $K$. Intuitively,  if we write $\underline{\mathbf A}$ for the N\'eron model of an abelian variety over $H_{d_1}$, and if we pretend that
$$
\mathcal Z(\mathbf A) = \sum_{\sigma \in \operatorname{Gal}(H_{d_1}/\Q)} \underline{\sigma(\mathbf{A})}
                      =\sum_{ \mathbf B \in Z(\mathbf A)} \underline{\mathbf{B}}
$$
is a well-defined integral model of $Z(\mathbf A)$ over $\Z$,
then the  normalized arithmetic intersection number is just $\mathcal Z(\mathbf A) \cdot \mathcal Z(m)$. When  $d_1=d_2 =d$ and $\mathfrak a=\co_{\kay}$, then $\mathcal Z(\mathbf A)$ is equal to the cycle $2\,\mathcal{C}\mathcal{M}(\co_\kay,  \alpha)$ of \cite[Page 18]{GRV}. Now we are ready to prove the main intersection formula, Theorem \ref{theo: MainFormula}, and restate it as the following theorem.

\begin{theorem}
\label{theo: mainIntSect}
Let the notation be as above, and assume  $\mathbf A \notin  Z(m)$, that is,  $a_\mathcal P(\frac{m}4, \frake_m) =0$. Define the rational numbers $a_p^{(0)}$ and $a_p^{(m)}$ via
\begin{align} \label{eq:bp}
 \sum_{\substack{\text{$p$ prime} \\ p|t_1}} a_p^{(0)} \log p &= -\sum_{\substack{\mu_1 + \mu_2 =\frake_{m} \in L'/L \\  0 \ne \mu_1 \in \mathcal N'/ \mathcal N  }}\kappa_{\mathcal N}(0, \mu_1) a_\mathcal P(m/4, \mu_2),
 \\
\sum_{\text{$p$ prime}} a_p^{(m)} \log p
     &= -\sum_{\substack{m_1 + m_2 = \frac{m}{4} \\  m_1>0,\, m_2 \ge 0 \\ \mu_1 + \mu_2 = \frake_{m} \in L'/L  }  }\kappa_{\mathcal N}(m_1, \mu_1) a_\mathcal P(m_2, \mu_2).  \notag
\end{align}
Then
\begin{align*}
\frac{\underline{\mathbf A} \cdot \mathcal Z(m)}{[K:H_{d_1}]}
&= \frac{h_{d_1} }{2 }\Bigg[   \sum_{p|t_1} a_p^{(0)}  \log p  +  \sum_{\text{$p|d_1$ or $p \le \frac{m |d_2 t|}4$ inert} } a_p^{(m)} \log p \Bigg]
\\
&=-\frac{h_{d_1} }{2 } \sum_{\substack{m_1 + m_2 = \frac{m}4, \, m_i \ge 0 \\ \mu_1 + \mu_2 = \frake_{m} \in L'/L \\  (m_1,\mu_1) \ne (0,0)  }  }\kappa_{\mathcal N}(m_1, \mu_1) a_\mathcal P(m_2, \mu_2).
\end{align*}
Here $h_{d_1}$   is  again the ring class number  of $\co_{d_1}$.  In particular, $\underline{\mathbf A}$ and $\mathcal Z(m)$ do not intersect in the fiber above $p$ unless $p|d_1$ or $p \le \frac{m|d_2t|}4$ is inert in $\kay$.
\end{theorem}

\begin{proof}
We consider the unique weakly holomorphic modular form
\begin{equation}
f_m (\tau) = q^{-m} + c(0) + c(1) q + \cdots  \in M_{-1/2}^{!, +}(\Gamma_0(4))
\end{equation}
and let $\Psi_m$ be its Borcherds lifting as in Theorem~\ref{thm:Borcherds}. Then  $(\mathcal Z(m), -\log\| \Psi_m\|^2)$ is an arithmetic divisor associated to the hermitean line bundle $\widehat{\mathcal M}_k$ with $ k=c(0)/2$, by Proposition~\ref{prop:integral}. Since $\mathbf A \notin \mathcal Z(m) (\C)$, the cycles $\underline{\mathbf{A}}$ and $\mathcal Z(m)$ intersect properly over $\co_K$.  This implies by  Lemma \ref{lem:inclusion} and \eqref{eq:divisor} that
\begin{equation} \label{eq:proper}
a_\mathcal P(\frac{m}4, \frake_{m}) =0.
\end{equation}
Formula (\ref{eq:FaltingsHeight}) implies
\begin{align*}
\frac{\underline{\mathbf A}  \cdot \mathcal Z(m)}{[K:\Q]}
 &= h_{\widehat{\mathcal M}_k}(\underline{\mathbf{ A}} ) + \frac{1}{[K:\Q]} \sum_{\sigma: K \hookrightarrow \C} \log\| \Psi_m( \mathbf A^\sigma)\|
 \\
 &= \frac{ c(0)}{2 }(h_{\Fal}(E_1) + h_{\Fal}(E_2)) + \frac{1}{2h_{d_1}}\log\| \Psi_m(Z(\mathbf A))\|.
\end{align*}
 According to Proposition \ref{prop5.3}, we have
\begin{equation}
h_{\Fal}(E_1) + h_{\Fal}(E_2) = \frac{1}2 ( C_0 + \kappa_{\mathcal N}(0, 0)),  \quad C_0 = \log 4 \pi + \Gamma'(1).
\end{equation}

 Recall that $\deg Z(\mathcal N) = 2 h_{d_1}$ and  $Z(\mathbf A) =Z(\mathcal N)$ by Theorem \ref{theo:Galois}.  By means of Theorems~\ref{thm:Borcherds} and~\ref{thm:SmallCM} we find
\begin{align*}
 &-4\log\| \Psi_m(Z(\mathbf A))\|_{\Pet}
 \\
 &= 2 h_{d_1} c(0) C_0 + \Phi(Z(\mathcal N), f_m)
 \\
  &= 2 h_{d_1}  c(0) ( C_0 +\kappa_{\mathcal N}(0, 0))    +
  2h_{d_1} \sum_{\substack{m_1 + m_2 = \frac{m}4, \, m_i\ge 0 \\ \mu_1 + \mu_2 = \frake_{m} \in L'/L \\  (m_1,\mu_1) \ne (0,0)} }\kappa_\mathcal N(m_1, \mu_1) a_\mathcal P(m_2, \mu_2).
\end{align*}
Consequently,
\begin{align*}
\frac{\underline{\mathbf A}  \cdot \mathcal Z(m)}{[K:H_{d_1}]}&=-\frac{h_{d_1}}2 \sum_{\substack{m_1 + m_2 = \frac{m}4, \, m_i\ge 0 \\ \mu_1 + \mu_2 = \frake_{m} \in L'/L \\  (m_1,\mu_1) \ne (0,0)} }\kappa_\mathcal N(m_1, \mu_1) a_\mathcal P(m_2, \mu_2)
\\
&= \frac{h_{d_1}}2 \bigg( \sum_{p |t_1} a_p^{(0)} \log p + \sum_{\text{$p$ prime}} a_p^{(m)} \log p \bigg) .
\end{align*}

Finally, we prove that the last sum is running only over primes with $p|d_1$ or $p \le \frac{m|d_2 t|}4$ inert in $\Q(\sqrt d)$. Assuming $p \nmid d_1$, it suffices to verify that  $a_p^{(m)} \ne 0$ implies $p \le \frac{m|d_2t|}{4}$. For each  $m/4 =m_1 +m_2$ in the sum, we have $0 < m_1 \le m/4$.
 Notice that $\mathcal N'\cong \frac{1}{\sqrt{d_1}} \bar{\mathfrak a}$ with quadratic form  $Q(x) = -\frac{t}{a} x \bar x$. As is seen in Theorem~\ref{theo4.6} (1), this implies
$m_1 = \frac{n}{|d_2t|}$ for some positive integer  $0 < n \le \frac{m |d_2t|}4$. If $\log p$ appears in some $\kappa_{\mathcal N}(m_1, \mu_1)$ with non-zero multiplicity,
we have  $o_p(n) =o_p(m_1) \ge 1$ by Theorem~\ref{theo4.6} (2). So $ p \le n \le \frac{m|d_2t|}4$ as claimed.
\end{proof}

\begin{corollary}
\label{cor:Special}
Let the notation be as at the beginning of this subsection,
and assume in addition that $d=d_1=d_2$ is fundamental. If  $\mathbf A \notin  Z(m)$, then
$$
\frac{\underline{ \mathbf  A} \cdot \mathcal Z(m)}{[K:H_{d}]}
=   - \frac{h_{d}}{2 }  \sum_{\substack{m_1 + m_2 = \frac{m}4,\, m_1 >0 \\ \mu_1 + \mu_2 = \frake_{m} \in L'/L} }\kappa_{\mathcal N}(m_1, \mu_1) a_\mathcal P(m_2, \mu_2) =\frac{h_{d_1}}{2 } \sum_{\text{$p$ prime}} a_p^{(m)} \log p.
$$
Here  $\kappa_{\mathcal N}(m_1, \mu_1)$ is given by (\ref{eq:kappa}).
\end{corollary}
\begin{proof}  In this case equation (\ref{eq:kappa0}) implies $a_p^{(0)}=0$.
\end{proof}

We remark that by Proposition \ref{prop:product} the abelian surfaces $A =E\times E$ considered in  \cite{GRV} are special cases of those in Corollary \ref{cor:Special}.

\begin{remark} Let the notation and hypotheses be as in Corollary~\ref{cor:Special}.  Then Theorem~1.2 of \cite{GR26} indicates that the cycle $Z(\mathbf A)$ should be equal to the special cycle $Z(T)$ in the sense of Kudla, where $T$ is a Gram matrix of the rank three positive definite lattice $\mathcal P=\mathcal P(\mathbf A)$ (but notice that our lattice $\mathcal P$ is slightly different from the refined Humbert invariant $q_{\mathbf A}$ of \cite{GR26}).
Hence \cite[Theorem 0.2]{KRSiegel} should imply that
$$
\frac{\underline{ \mathbf  A} \cdot \mathcal Z(m)}{[K:H_{d}]}
\doteq
(\mathcal Z(T)\cdot \mathcal Z(m))^{\text{proper}} \doteq \sum_{S= \kzxz {T} {a} {{}^t a} {m} >0} E'_S(\tau, 0) e^{-2\pi i\tr(S\tau)}
$$
in the notation of  \cite{KRSiegel},  where  $E_S'(\tau, 0)$ is the $S$-th coefficient of the  central derivative of a certain incoherent Eisenstein series $E(\tau, s)$  on $\Sp_{8}$ defined in  \cite{KRSiegel}. The sum on the right hand side runs over (the finitely many)  $a\in \frac{1}{2}\Z^3$ such that the matrix $S$ is positive definite, that is, ${}^taT^{-1}a<m$. The dots over the equality signs indicate that minor modifications may be needed to obtain  actual equalities.  Comparing this with Corollary \ref{cor:Special} suggests some intriguing relation between the incoherent Eisenstein series on $\Sp_8$ and the incoherent Eisenstein series on $\SL_2$ in Section \ref{sect:inceis} of this paper.
\end{remark}

\section{Positivity of the normalized arithmetic intersection} \label{sect:positive}

In this section, we prove Theorem \ref{theo:positive}. Our method of proof can be viewed as a variant of the approach of \cite{Li21}. The idea is to use the CM value formula for higher automorphic Green functions of \cite{BEY} to show that the right hand side of the formula of Theorem \ref{theo: MainFormula} is positive (and each term in the sum is non-positive).

We begin by recalling some facts on higher automorphic Green functions in our setting. In particular, we consider the even lattice $L$ of signature $(3,2)$ of Section \ref{sect:siegel3fold} and its associated orthogonal Shimura variety which is isomorphic to $\mathcal{A}_2(\C)$.
We refer to Sections~4 and~5 of \cite{BEY} for a more detailed introduction in greater generality.

Let $j\in \Z_{\geq 0}$ and consider $f\in M^{!,+}_{-1/2-2j}$.
Denote by $\vec f$ the associated vector valued  weakly holomorphic modular form for the dual Weil representation of $L$ as in  Proposition~\ref{prop:scalvec}.
We write
\begin{align*}
R_k  &=2i\frac{\partial}{\partial\tau} + k v^{-1}
\end{align*}
for the Maass raising operator taking (possibly non-holomorphic) modular forms of weight $k$ to forms of weight $k+2$, and
$R_k^j  = R_{k+2(j-1)}\circ\dots\circ R_k$ for its $j$-fold iteration taking forms of weight $k$ to forms of weight $k+2j$.
Denote by
\begin{align}
\Phi^j(z,f):= \frac{1}{(4\pi)^j}\Phi(z,R^j_{-1/2-2j}\vec f\, )
\end{align}
the higher regularized theta lift of $f$ defined in \cite[Section~5]{BEY}. For $j=0$, this is the regularized theta lift appearing in Theorem \ref{thm:Borcherds} (iii). If $f=f_m$ for a  discriminant $m\in \Z_{>0}$, then up to a  constant it can also be described as the constant term in the Laurent expansion of the automorphic Green function
\begin{align}
\label{eq:phim}
\Phi_{m}(z,s) &=2\frac{\Gamma(s+\frac{1}{4})}{\Gamma(2s)}
\sum_{\substack{x\in \frake_m+L\\ Q(x)=\frac{m}{4}}}
\left(\frac{m}{4Q(x_{z^\perp})}\right)^{s+\frac{1}{4}}
F\left(s+\frac{1}{4},s-\frac{1}{4},2s;\frac{m}{4Q(x_{z^\perp})}\right)
\end{align}
at the value $s_0=5/4$ of the spectral parameter $s$.
Here $F(a,b,c;t)$ denotes the Gauss hypergeometric function as in \cite[Chapter 15]{AS} and $x_{z^\perp}$ the orthogonal projection of $x\in L'$ on the orthogonal complement of the  negative two-plane corresponding to $z$.
Note that the series converges absolutely and locally uniformly for $\Re(s)>s_0$. For real $s$ it takes values in $\R\cup \{\infty\}$.

If $j>0$ and $f_m=q^{-m}+O(1)\in M^{!,+}_{-1/2-2j}$, Proposition~4.7 of \cite{BEY} implies that $\Phi^j(z,f_m)$ is
given by the value of the automorphic Green function at $s=s_0+j$ as
\begin{align*}
\Phi^j(z,f_m)= (\tfrac{m}{4})^j j! \,\Phi_m(z,s_0+j).
\end{align*}
For more general $f\in M^{!,+}_{-1/2-2j}$ with Fourier coefficients $c(n)$,  we have
\begin{align*}
\Phi^j(z,f)= j! \sum_{m\in \Z_{>0}}c(-m)(\tfrac{m}{4})^j  \, \Phi_m(z,s_0+j).
\end{align*}
Here we will only require the special case where $j=1$, in which we obtain
\begin{align}
\label{eq:ghfor}
\Phi^1(z,f_m)&= \frac{m}{4} \cdot\Phi_m(z,\frac{9}{4})\\
\nonumber
&=\frac{m\Gamma(\frac{5}{2})}{2\Gamma(\frac{9}{2})}
\sum_{\substack{x\in \frake_m+L\\ Q(x)=\frac{m}{4}}}
\left(\frac{m}{4Q(x_{z^\perp})}\right)^{\frac{5}{2}}
F\left(\frac{5}{2},2,\frac{9}{2};\frac{m}{4Q(x_{z^\perp})}\right)
\end{align}
for $f_m=q^{-m}+O(1)\in M^{!,+}_{-5/2}$ as in Remark~\ref{rem:4.4}.
We now adapt the argument of \cite{Li21} to derive lower bounds for CM values of this function.

We denote by $Z_\eps(m)$ the $\eps$-neighborhood of the divisor $Z(m)$ with respect to the invariant metric on $\D^+\cong \H_2$ defined by
\begin{align*}
Z_\eps(m) &= \bigcup_{\substack{x\in \frake_m+L\\ Q(x)=\frac{m}{4}}}
\{ z\in \D\mid \; \tfrac{4}{m} |Q(x_z)|<\eps\}.
\end{align*}
Note that for vectors $x\in L'$ with  $Q(x)=\frac{m}{4}$ the condition $ \tfrac{4}{m} |Q(x_z)|<\eps$ is equivalent to
\begin{align}
\label{eq:lbound}
\frac{1}{1+\eps}<\frac{m}{4Q(x_{z^\perp})}.
\end{align}
Moreover, the right hand side is bounded by $1$. If $\eps<\eps'$, then $Z_\eps(m)\subset Z_{\eps'}(m)$, and
\[
\bigcap_{\eps>0} Z_\eps(m) = Z(m), \qquad \bigcup_{\eps>0}  Z_\eps(m) =\D.
\]

\begin{theorem}
Let $\eps>0$, and let $Z(\calN)$ be the small CM cycle associated with $\calN$ defined in \eqref{eq:smallcm}. Assume that $Z(\calN)$ and $Z(m)$ are disjoint. Then
\begin{align*}
\Phi^1(Z(\calN),f_m)
&\geq
\frac{2m}{35}
\cdot |Z(\calN)\cap Z_\eps(m)|
(1+\eps)^{-\frac{5}{2}}
F\left(\frac{5}{2},2,\frac{9}{2};\frac{1}{1+\eps}\right).
\end{align*}
\end{theorem}

\begin{proof}
We use the absolutely convergent series \eqref{eq:ghfor} to derive a lower bound for the CM value.
The power series expansion of the Gauss hypergeometric function implies that
\[
F\left(\frac{5}{2},2,\frac{9}{2}; t \right) \geq 1+ \frac{10}{9} t
\]
for $t\in [0,1)$. Moreover, it is monotonically increasing in this interval.

Let $z\in \D$ be a point which maps to an element of $Z(\calN)\cap Z_\eps(m)$ under the quotient map. Then there exists $x^*\in \frake_m+L$ with  $Q(x^*)=\frac{m}{4}$ and $\tfrac{4}{m} |Q(x_{z}^*)|<\eps$. Omitting all terms in the sum over $x$ in  \eqref{eq:ghfor} except for the contribution of $x^*$ we obtain the lower bound
\begin{align*}
\Phi^1(z,f_m)
&\geq
\frac{2m}{35}
\left(\frac{m}{4Q(x^*_{z^{\perp}})}\right)^{\frac{5}{2}}
F\left(\frac{5}{2},2,\frac{9}{2};\frac{m}{4Q(x^*_{z^{\perp}})}\right)\\
&\geq\frac{2m}{35}
(1+\eps)^{-\frac{5}{2}}
F\left(\frac{5}{2},2,\frac{9}{2};\frac{1}{1+\eps}\right).
\end{align*}
Adding the contributions of all elements of $Z(\calN)\cap Z_\eps(m)$, we get the assertion.
\end{proof}

\begin{corollary}
\label{cor:phi1pos}
The CM value $\Phi^1(Z(\calN),f_m)$ is strictly positive.
\end{corollary}

\begin{proof}
This follows from the above theorem by choosing $\eps>0$ such that $Z(\calN)\cap Z_\eps(m)$ is non-empty.
\end{proof}

We now recall some facts on Rankin-Cohen brackets, see e.g.~\cite[Section 3.1]{BEY} for more details.
Let $L_1$ and $L_2$ be even lattices and $k,l\in \frac{1}{2}\Z$.
For $f\in M_{k,\rho_{L_1}}$ and $g\in M_{l,\rho_{L_2}}$, the $j$-th Rankin-Cohen bracket is defined by
\begin{align}
\label{eq:defRC}
[f,g]_j = \sum_{s=0}^j (-1)^s \binom{k+j-1}{s} \binom{l+j-1}{j-s} f^{(j-s)} \otimes g^{(s)},
\end{align}
where $f^{(s)}:= \frac{1}{(2\pi i)^s}\frac{\partial^s}{\partial \tau^s} f$. It is a vector valued modular form of weight $k+l+2j$ for the Weil representation $\rho_{L_1\oplus L_2}\cong\rho_{L_1}\otimes \rho_{L_2} $. We will also use the above formula to extend the definition of the Rankin-Cohen bracket to vector valued mock modular forms. For small $j$ we have
\begin{align}
\label{eq:j01}
[f,g]_0 = f\otimes g, \qquad [f,g]_1 = l f^{(1)}\otimes g - kf\otimes g^{(1)}.
\end{align}
We are now ready to state the CM value formula \cite[Theorem~5.4]{BEY} for higher Green functions
in the setting of the present paper.

\begin{theorem}
\label{thm:fund}
Let $f\in M^{!,+}_{-1/2-2j}$.
The value of the higher Green function
$\Phi^j(z,f)$ at the CM cycle $Z(\calN)$ is given by
\begin{align*}
\Phi^j(Z(\calN),f)&=\deg(Z(\calN))\cdot
\operatorname{CT}\big(\langle
\vec f,\, [\theta_{\calP},\calE_{\calN}]_j\rangle\big).
\end{align*}
\end{theorem}

For $j=0$ this result reduces to Theorem~\ref{thm:SmallCM}. We now make use of it in the case $j=1$. Recall the Fourier expansions of $\theta_{\calP}$ and $\calE_{\calN}$ as in Section~\ref{sect:4.3}.

\begin{proposition}
\label{prop:nonpos}  The following are true.
\begin{enumerate}

\item For $m_1> 0$ or  $0\ne \mu_1\in \calN'/\calN$ we have $\kappa_\calN(m_1, \mu_1)\leq 0$.

\item  For $m_2\geq 0$ and $\mu_2\in \calP'/\calP$ we have $a_\calP(m_2, \mu_2)\geq 0$.

\item  For all primes $p$, we have $a_p^{(m)} \ge 0$ and $a_p^{(0)} \ge 0$.


\end{enumerate}
\end{proposition}

\begin{proof}  When  $m_1>0$, then  (1) is a special case of
\cite[Proposition~3.1]{Li21} with $F=\Q$.

The case $m_1=0$ and $\mu_1 \ne 0$ can be proved similarly.
Indeed, let
$$
f(s)  = -\prod_{p|d, \,p\nmid t_1} \delta_{0, \mu_{1, p}} \prod_{p |t_1} \tilde{W}_{p}(1, s, \phi_{\mu_1}) .
$$

Since $\mu_1 \ne 0$ and $E_0(\tau, 0,  \phi_{\mu_1}) =0$, we have $f(0) =0$, and there is at least one prime $p|t_1$ such that $\tilde{W}_{p}(1, 0, \phi_{\mu_1})=0$ (otherwise $f(s) =0$). So
$$
\kappa_{\mathcal N}(0, \mu_1) = f'(0)
 = -\tilde{W}_{p}'(1, 0, \phi_{\mu_1}) \prod_{\substack{\text{$q$ prime }\\ q|t_1,\, q\ne p}} \tilde{W}_{q}(1, 0, \phi_{\mu_1}) \prod_{\substack{\text{$q$ prime} \\q|d,\, q\nmid t_1}} \delta_{0, \mu_{1,q}}.
$$
Looking at \cite[Propositions 5.1-5.4]{LYY}, we see $\tilde{W}_{q}(1, 0, \phi_{\mu_1}) \ge 0$.
When  $p$ is split, $1/L_p(s, \chi_d)$ has a zero at $s=0$, and
$$
\tilde{W}_{p}'(1, 0, \phi_{\mu_1}) =L_p(1, \chi_d) \frac{W_{p}(1, 0, \phi_{\mu_1})}{\gamma(t) \sqrt{|d|}} \log p \ge 0
$$
by the same propositions.

When  $p$ is non-split,
$$
\tilde{W}_{p}'(1, 0, \phi_{\mu_1})= \frac{L_p(1, \chi_d)}{L_p(0,\chi_d)} \frac{W_{p}'(1, 0, \phi_{\mu_1})}{\gamma(t) \sqrt{|d|}} \ge 0
$$
by \cite[Propositions 5.3, 5.4]{LYY}. So $\kappa_{\mathcal N}(0, \mu_1) = f'(0) \le 0$.

Assertion (2) is clear as $\mathcal P$ is positive definite. Finally,  (3) follows from  (1) and (2).
\end{proof}

\begin{proposition}
\label{prop:fund}
(i) Assume that $Z(\calN)$ and $Z(m)$ are disjoint.
Let $f_m\in M^{!,+}_{-5/2}$ be the unique form whose Fourier expansion begins as $f_m=q^{-m}+O(1)$.
Then
\begin{align*}
\Phi^1(Z(\calN),f_m)&=\deg(Z(\calN))
\sum_{\substack{\mu_1 \in \mathcal N'/\mathcal N\\ \mu_2 \in \mathcal P'/\mathcal P \\
\mu_1+\mu_2\equiv \frake_{m}\,(L)}}\sum_{\substack{m_1>0\\m_2\geq 0\\ m_1 + m_2 = \frac{m}{4}}}
	(m_2-\frac{3}{2}m_1)a_\mathcal P(m_2, \mu_2)\kappa_{\mathcal N}(m_1, \mu_1) .
\end{align*}
(ii) There exist $m_1,m_2\in \frac{1}{4|d_2|t}\Z$ with $m_1>0$, $m_2\geq 0$ and $m_1 + m_2 = \frac{m}{4}$, as well as $\mu_1 \in \mathcal N'/\mathcal N$ and $\mu_2 \in \mathcal P'/\mathcal P$ such that  $\mu_1+\mu_2\equiv \frake_{m}\,(L)$, for which
%
\[
(m_2-\frac{3}{2}m_1)a_\mathcal P(m_2, \mu_2)\kappa_{\mathcal N}(m_1, \mu_1)>0.
\]
\end{proposition}

\begin{proof}
According to  \eqref{eq:j01} the Fourier expansion of $[\theta_{\calP},\calE_{\calN}]_1$ is given by
\begin{align*}
\sum_{\substack{\mu_1 \in \mathcal N'/\mathcal N\\ \mu_2 \in \mathcal P'/\mathcal P \\ n\geq 0}}
 \sum_{\substack{m_1\geq 0\\m_2\geq 0\\ m_1 + m_2 = n}} (m_2-\frac{3}{2}m_1)a_\mathcal P(m_2, \mu_2)\kappa_{\mathcal N}(m_1, \mu_1) q^n\phi_{\mu_1}\otimes \phi_{\mu_2}.
\end{align*}
Notice that the constant term vanishes. Hence, in the pairing $\operatorname{CT}\big(\langle
\vec f_m,\, [\theta_{\calP},\calE_{\calN}]_1\rangle\big)$ in Theorem~\ref{thm:fund} only the Fourier coefficients of index $\frac{m}{4}$ with $\mu_1+\mu_2\equiv \frake_{m}\,(L)$ give a non-vanishing contribution. Since $Z(\calN)$ and $Z(m)$ are disjoint, the coefficients  $a_\calP(m,\mu_2)$ with $\mu_2\equiv \frake_{m}\,(L)$ vanish, and we only need to sum over $m_1>0$. This leads to the formula claimed in (i).

For (ii), we use in addition the fact that  $\Phi^1(Z(\calN),f_m)>0$ by Corollary~\ref{cor:phi1pos}. Consequently, at least one of the summands in the formula of (i) must be positive.
\end{proof}

\begin{remark}
Because of Proposition~\ref{prop:nonpos}, the pair $(m_1, m_2)$ in  Proposition \ref{prop:fund} (ii) must satisfy $m_1> \frac{2}{5}m$.
\end{remark}

\begin{proof}[Proof of Theorem \ref{theo:positive}] By Proposition \ref{prop:nonpos}, all terms $\kappa_{\mathcal N}(m_1, \mu_1) a_{\mathcal P}(m_2, \mu_2) \le 0$. By Proposition \ref{prop:fund} and the above remark, we see at least one
$$
\kappa_{\mathcal N}(m_1, \mu_1) a_{\mathcal P}(m_2, \mu_2) < 0.
$$
Now the theorem is clear.
\end{proof}

\section{Bad reduction of genus two curves with small CM} \label{sect:BadReduction}

We recall that a curve $C$ of genus $\geq 2$ over a number field $K$  is said to have  stable  bad reduction at a rational prime $p$, if the stable model of $C$
(over some finite extension of $K$)
has non-smooth special fiber at some prime above $p$. Equivalently, there exists a prime above $p$, at which $C$ does not have potentially good reduction.
%
For background on the reduction of algebraic curves and on stable models we refer to \cite[Chapter 10]{Liu-book}.

It is well-known that a principally polarized abelian surface is not isomorphic to the Jacobian of a genus two curve if and only if it is decomposable, that is, isomorphic to the product of two elliptic curves equipped with the product polarization, see e.g.~\cite[Proposition~6]{Ka}.
Let $C$ be a genus two curve over a number field $K$ with CM by a biquadratic field, and assume that its Jacobian $\mathbf J(C) =(J(C), \theta_C)$ is isomorphic to some $\mathbf A=(A(\mathfrak a, t),  \lambda)$ as in
Section~\ref{sect:5.2}.    By enlarging the field of definition $K$ if necessary, we may assume that $\mathbf J(C)$ has good reduction everywhere, i.e., its N\'eron model defines a point
$\underline{\mathbf J(C)}
\in \mathcal A_2 (\co_K)$.
Then the curve $C$ has  stable bad reduction at a rational prime $p$
if and only if $\underline{\mathbf J(C)}$ modulo some prime $\mathfrak p$ of $K$ above $p$ is isomorphic to the  product of two elliptic curves with the product polarization, that is,  $\underline{\mathbf {J}(C)}$ and $\mathcal A_1 \times \mathcal A_1 $ have non-zero intersection  at $p$.
Hence,  Theorem~\ref{theo: mainIntSect}  implies the following result.

\begin{theorem}
\label{theo:Curve} Let the notation and assumptions be as above.
\begin{enumerate}
\item We have
$$
\frac{\underline{\mathbf  J(C)} \cdot (\mathcal A_1 \times \mathcal A_1)}{[K:H_{d_1}]}
= \frac{h_{d_1} }{2 }\bigg( \sum_{p|t_1} a_p^{( 0)}  \log p  + \sum_{\text{$p|d$ or $p \le \frac{|d_2 t|}4$ inert} }  a_p^{(1)} \log p \bigg).
 $$


\item The curve  $C$ has stable bad reduction at a rational prime $p$ if and only if
$$
 a_p^{(0)}+  a_p^{(1)} \neq 0 .
$$
In  particular,  if $C$ has stable bad reduction at some rational prime $p$, then  $p \le \frac{|d_1|}4$ and $p| d_1$ or $ p$  is inert in $\mathbb Q(\sqrt{d_1})=\mathbb Q(\sqrt{d_2})$. 
Moreover, the stronger bound 
$p \le \frac{|d_2| t}4$ holds unless $d_2=-3$ and $t=\ell$ is a prime number 
that is ramified or split in $\Q(\sqrt{-3})$. 

\item  When  $d_1=d_2=d$, we have $a_p^{(0)}=0$. Moreover, $\kappa_{\mathcal N}(m_1, \mu_1)$ in the sum defining $a_p^{(1)}$ is given explicitly by (\ref{eq:kappa}).
\end{enumerate}
\end{theorem}


\begin{proof}
Everything follows directly from Theorem~\ref{theo: mainIntSect} except for  the claim that $p \le \frac{|d_2| t}4$    when  $p|d_1$.
If  $p >  \frac{|d_2| t}4=\frac{|d| t_2^2  t}4$ and $p|d_1=d t_2^2 t^2$, then we have the following cases:
\begin{enumerate}

\item  $d_1 =d= -3, -4$,

\item $d_1=-12$ and $d_2=-3$,

\item $d_1=d=  -\ell, -3\ell$ with $\ell> 3$, $t=1$.

\item $d_2 =d=-\ell$ with $\ell\ge 5$ and $t=2, 3$,

\item $d_2=-3$ and  $t=\ell \ge  3$ is inert in $\Q(\sqrt{-3})$, and

\item $d_2=-3$ and  $t=\ell \ge  3$ is ramified or split in $\Q(\sqrt{-3})$.
\end{enumerate}
Here $\ell$ denotes a prime number and in all cases $p=\ell$.
Cases (1) and (2) never happen, as there are no non-product principal polarizations on $E_{d} \times E_{d}$ for $d=-3, -4$, and   on $E_{-12} \times E_{-3}$, where $E_d =\C/\mathcal O_d$, see \cite[Theorem 2]{Kani14}.  

For the cases (3)--(6), we first verify that  $\log p$ does not appear in $\kappa(m_1, \mu_1)$ for $ 0<  m_1 \le  \frac{1}4$ and thus $a_p^{(1)}=0$.  
Indeed,
$ \mu_1  \in \frac{1}{\sqrt{d_2} t} \bar{\mathfrak a}/\bar{\mathfrak a}$ and
 $m_1 \equiv  -\frac{t}a \mu_1 \bar \mu_1  \pmod 1$. Hence $m_1 \in \frac{1}{|d|t} \Z_{\ge 0}$ as $d_2 =d$ in these cases. If $\ord_p m_1 \ge 0$,  then $ m_1 \in \frac{1}3 \Z_{\ge 0}$ or $\frac{1}2 \Z_{\ge 0}$.  This is impossible as we assume $0 < m_1 \le  \frac{1}4$. 
 Now assume that $\ord_p(m_1) <0$. 
If a $\log(p)$-term  appears in $\kappa(m_1,\mu_1)$, then the local Whittaker function  $W_{m_1, p}(1, s, \phi_{\mu_1})$ vanishes at $s=0$, and
 $$
 \kappa(m_1, \mu_1) =c_\infty  W_{m_1, p}'(1, 0, \phi_{\mu_1}) \prod_{ q \ne p} W_{m_1, q}(1, s, \phi_{\mu_1})|_{s=0}
 $$
for some explicit non-zero constant $c_\infty$. The local Whittaker functions are calculated in the appendix of \cite{YYY}. In our case, $(\mathcal N_p, Q) \cong ( \mathcal O_{\kay, p},  \beta_p x \bar x)$ with $\beta_p= -\frac{t z_p \bar{z}_p}a$, and $\mathfrak a_p =\mathcal O_p z_p$.   Now Propositions 5.3 and 5.4 of \cite{YYY} imply that either $W_{m_1, p}(1, 0, \phi_{\mu_1}) \ne 0$ or $W_{m_1, p}(1, s, \phi_{\mu_1}) =0$ for all $s$.
We leave the details to the reader.  Hence, $\log p$ does not contribute to $\kappa(m_1, \mu_1)$, and  $a_p^{(1)}=0$ in all cases.

We now consider the quantity $a_p^{(0)}$.
For  the cases (3) and (4), Theorem~\ref{theo4.6}(2) implies that $a_p^{(0)} =0$ because $p\mid d$ and $p\nmid t$. This settles these cases.

   In case (5), we can take $\mathfrak a =\mathcal O_{-3}$ as $\kay=\Q(\sqrt{-3})$ has class number $1$.   Theorem~\ref{theo4.6}(2) (with $d=-3$ and  $  t_1 =\ell =p$) asserts that 
   $$
\kappa_{\mathcal N}(0, \mu_1) =  - \delta_{0, \mu_{1,3}}  \cdot \frac{d}{ds}\Big|_{s=0} \tilde{W}_{0, p}(1, s, \phi_{\mu_{1, p}})
$$
with
$$
\tilde W_{0, p} (1, s, \phi_{\mu_{1, p}}) = \frac{L_p(s+1, \chi)}{L_p(s, \chi)} \cdot \frac{W_{0, p}(1, s, \phi_{\mu_{1, p}})}{\gamma(\mathcal C_p) |d|_p^{\frac{1}2}}.
$$
Since  $p=\ell$ is inert in $\Q(\sqrt{-3})$,  we have   $\frac{L_p(1, \chi)}{L_p(0, \chi)}\ne 0$.

If  $\mu_{1, 3} \ne 0 $ in $ \mathcal N_3'/\mathcal N_3$, then $\kappa_{\mathcal N}(0, \mu_1) =0$.  If  $\mu_{1, 3} =0 $ in  $ \mathcal N_3'/\mathcal N_3$, we have $\mu_{1, p}\ne 0 $ in $\mathcal N_p'/\mathcal N_p$.  
So  $\ord_{p}\mu_{1, p} =-1$ and $\ord_p(p \mu_{1, p}\bar\mu_{1, p})=-1$. Now Corollary~5.6 of \cite{YYY} implies that $\kappa_{\mathcal N}(0, \mu_1) =0$ and therefore $a_p^{(0)}=0$.

In summary, except for the case (6), we optain that $p \le \frac{|d_2| t}4$ when  $C$ has stable bad reduction at $p$.  In the exceptional case (6), we still have
$p \le \frac{|d_1|}4 =\frac{|d| p^2}4$.  The claim that $p|d_1$ or $p$ is inert in $\Q(\sqrt{d_1})$ follows directly from Theorem ~\ref{theo: mainIntSect}.
\end{proof}

\begin{remark} 
\label{rem7.2} 
Here we provide some examples showing that the sharper bound $p\leq \frac{|d_2|t}{4}$ of Theorem \ref{theo:Curve}(2) is violated in the exceptional cases where $d_2=-3$ and $t=\ell$ is a prime number that is ramified or split in $\Q(\sqrt{-3})$. 

In the ramified case, when $t=3$, we take take $\alpha =\kzxz {2} {3} {1}  {2}$. Then a computation shows that $a_3^{(0)}=\frac{2}3  \ne 0$. The associated genus two curve $C: y^2 =2x^6 + 5x^3+2$ has stable bad reduction at $p=3 > \frac{9}4$.  

In the split case, when $t=\ell >3$ is split in $\Q(\sqrt{-3})$, we have $\chi_p=1$, and
$$
\tilde W_{0, p} (1, s, \phi_{\mu_{1, p}}) = \frac{1-p^{-s}}{1-p^{-1-s}} \cdot \frac{W_{0, p}(1, s, \phi_{\mu_{1, p}})}{\gamma(\mathcal C_p) |d|_p^{\frac{1}2}}.
$$
Recall that  $0 \ne \mu_1 \in \mathcal N'/\mathcal N= \mathcal N_3'/\mathcal N_3 \oplus \mathcal N_p'/\mathcal N_p$.  If $\mu_{1, 3} \ne 0 $, we have $a_p^{(0)} =0$ by Theorem~\ref{theo4.6}(2). If  $\mu_{1, 3} = 0$, we have $\mu_{1,p} \ne 0$. We identify $\mathcal N_p$ with $\Z_p \oplus \Z_p$ with quadratic form $Q(\mu', \mu'') = p \mu' \mu''$. Choose a representative  $\mu_{1, p} =(\mu', \mu'') \in \mathcal N_p'\setminus\mathcal N_p$, then at least one of $\mu'$ and $\mu''$ is not in $\Z_p$.  The local Whittaker function 
$W_{0, p}(1, s, \phi_{\mu_{1, p}})$ is a polynomial in $p^{-s}$, and
$$
\tilde W_{0, p}' (1, 0, \phi_{\mu_{1, p}})=\frac{p}{p-1} \log p \cdot \frac{W_{0, p}(1, 0, \phi_{\mu_{1, p}})}{\gamma(\mathcal C_p) |d|_p^{\frac{1}2}}.
$$
Now \cite[Corollary 5.6]{YYY} implies that 
$$
\kappa_\mathcal N(0, \mu_{1}) =\begin{cases}
  -\frac{1}{p-1} \log p &\hbox{if  exactly one of } \mu', \mu'' \in \Z_p,
  \\
  0  &\mu', \mu'' \notin \Z_p.
\end{cases}
$$
Therefore
\begin{equation}
a_p^{(0)} =\frac{1}{p-1}\sum_{ \substack{\mu_1 + \mu_2 =\mathfrak e_1 \in L'/L \\
    \text{$\mu_1 \in \mathcal N'/\mathcal N$ satisfies $(*)$}}} a_{\mathcal P}(\frac{1}4, \mu_2),
\end{equation}
where $(*)$ denotes  the following condition:
$$
\mu_{1, 3} =0, \quad    \mu_{1, p} =(\mu', \mu'') \in \frac{1}p (\Z_p \oplus  \Z_p)/(\Z_p \oplus  \Z_p) , \hbox{ and  exactly one of } \mu', \mu'' =0.
$$
One such  example is $d_2=-3$,  $t=7$, with $\alpha = \kzxz {2} {7} {1} {4}$. A computation shows that $a_7^{(0)}=\frac{2}3 \ne 0$. The associated genus two curve $C$ has the equation
$$
C: y^2 = (x^2 +7 \sqrt{21} -48)^3 + 8000 \cdot (28-6 \sqrt{21}).
$$
We thank Harun Kir for providing the above examples with $d_2 =-3$ and $\ell =3, 7$ to us, which helped to correct an error in the exceptional case of of Theorem \ref{theo:Curve}(2) in an earlier version of the present paper.  
\end{remark}

\begin{proof}[Proof of Theorem \ref{thm:BadReduction}]
By  Proposition  \ref{prop:product}, we have
$$
\mathbf J(C) \cong (E_1 \times E_2, \lambda) \cong  (\C/\co_{l} \times \C/\mathfrak a, \lambda)
$$
with $l=-\operatorname{lcm}(|d_1|, |d_2|)$ and $\co_{\mathfrak a} =\co_{g}$ with $g =-\gcd(|d_1|, |d_2|)$.  Now Theorem \ref{thm:BadReduction}  follows from Theorem \ref{theo:Curve}(2).
\end{proof}



\begin{proof}[Proof of Theorem \ref{theo:AlwaysBad}]  By \cite[Corollary 10.6.3]{BL}  and \eqref{eq:canonical}  we may assume that  $\mathbf J(C) \cong  (\C/\co_{d_1} \times \C/\mathfrak a, \lambda) $.     Theorem  \ref{theo:positive} implies that
$$
\frac{\underline{\mathbf  J(C)} \cdot (\mathcal A_1 \times \mathcal A_1)}{[K:H_{d_1}]} >0.
$$
Hence the local arithmetic intersection is positive for at least one prime $p$, i.e., $C$ has stable bad reduction at
$p$.
\end{proof}

To derive Corollary \ref{prop:Conjecture}, we need the following  counting lemma.

\begin{lemma}
\label{lem:explicit-formula}
Assume that $d=-\ell$ for a prime $\ell\equiv 3\pmod 4,~\ell>3$. For $m\in\frac{1}{\ell}\Z_{>0}$ with $m \le \frac{1}4$, and  $\mu\in\mathcal N'/\mathcal N$ with $Q_{\mathcal N}(\mu)\equiv m\pmod{1}$, we have
$$\frac{h_d}{2}\kappa_{\mathcal N}(m,\mu)=F(m\ell).$$
Here, for any positive integer $N$,
$$
F(N) = \sum_{n|N} (\frac{-\ell}n) \log n.
$$
\end{lemma}
\begin{proof}
Recall from Proposition \ref{prop:N} that $(\mathcal N,Q_{\mathcal N})\cong (\overline{\mathfrak a},-\frac{1}{a}x\bar{x})$.
Let
\begin{align*}
\operatorname{Diff}(m, \mathcal N) &= \{ \text{$q<\infty$ prime}:\;  \mathcal N \otimes_\Z \Q_q \text{ does  not represent } m\}\\
&=\{ \text{$q<\infty$ prime}:\;  \chi_q(-m/a) =-1\}
\end{align*}
be the `Diff set' defined by Kudla in \cite{KuAnnal}. Then  $|\hbox{Diff}(m, \mathcal N)|\ge 1 $  odd, and for any $q \in \hbox{Diff}(m, \mathcal N)$, the local Whittaker function $W_{m, q}(0, \phi_\mu) =0$ for any $\mu$. So $\kappa_{\mathcal N}(m, \mu)=0$ unless $\hbox{Diff}(m, \mathcal N) =\{ q\}$  for a single prime  $q$, which has to be non-split.  We first verify that $\ell \notin \hbox{Diff}(m, \mathcal N)$. Indeed,  by assumption,   $N=m\ell \le \frac{\ell}4$ is prime to $\ell$.  The condition  $Q_{\mathcal N}(\mu)\equiv m\pmod{1}$ implies that
$N \equiv Q_\mathcal N(\mu \sqrt{-\ell}) \pmod  \ell$, and that $\mathcal N\otimes_\Z \Q_\ell$ represents $N$. So $\ell \notin \hbox{Diff}(m, \mathcal N)$.

Now assume that $\hbox{Diff}(m, \mathcal N) =\{ q\}$ with $q$ inert in $\kay=\Q(\sqrt{-\ell})$, and write $N=m \ell $.  In this case, $e =\ord_q m=\ord_q N$ is odd, and  (\ref{eq:kappa}) implies
$$
-h_{d} \kappa_{\mathcal N}(m, \mu)
=(e +1) \rho(N/q) \log q.
$$
On the other hand, write $N=q^e N_1$ and recall that $\chi(n) =(\frac{-\ell}{n})$ is the quadratic character associated to $\kay/\Q$. Noticing that $\chi(q) =-1$,  we find    as in \cite[Remark 4.1]{YY19} that
\begin{align*}
F(N)&=\sum_{n_1|N_1}\sum_{j=0}^{e}\chi(q)^{j}\chi(n_1)(j\log q +\log n_1)
\\&=\bigg(\sum_{j=0}^{e}(-1)^{j}j\bigg)\bigg(\sum_{n_1|N_1}\chi(n_1)\bigg)\log q+\bigg(\sum_{j=0}^{e} (-1)^{j}\bigg)F(N_1)
\\
&=-\frac{e+1}{2}\rho(N_1)\log q =-\frac{e+1}{2}\rho(N/q)\log q.
\end{align*}
This proves the lemma.
\end{proof}

\begin{proof}[Proof of Corollary \ref{prop:Conjecture}]  The case $\ell=3$ does not appear as there is no genus two curve $C$ with $J(C) \cong \C/\mathcal O_{-3} \times \C/\mathcal O_{-3}$. We assume $\ell>3$ from now on.
By  Theorem~\ref{theo:Curve} and Lemma~\ref{lem:explicit-formula}, we have
\begin{align*}
\frac{\underline{\mathbf J(C)} \cdot (\mathcal A_1 \times \mathcal A_1)}{[K:H_d]}
&=-\frac{h_d}{2}\sum_{\substack{x\in\mathcal P'\\
\ell (\frac{1}4 - Q(x))\in\Z_{>0}}}\kappa_{\mathcal N}\left(\frac{1}{4}-Q(x), \mathfrak e_1-x\right)\\
&=-\sum_{\substack{x\in\mathcal P'\\\frac{\ell -4\ell Q(x)}{4}\in\Z_{>0}}}F\left(
\frac{\ell -4\ell Q(x)}{4}\right)
\end{align*}
as claimed.
The last identity follows from  Lemma \ref{lem:explicit-formula} and the claim that    for any $x\in\mathcal P'$ such that $\ell (\frac{1}4 - Q(x))\in\Z_{>0}$, we have  $\mathfrak e_1-x \in\mathcal N'/\mathcal N$ and $Q_{\mathcal N}(\mathfrak e_1-x)\equiv\frac{1}{4}-Q(x)\pmod{1}$. The latter claim  is actually a corollary of Proposition \ref{prop:fqm2}.
Indeed,  since $|\mathcal P'/\mathcal P|=2 \ell = |L'/L| \cdot |\mathcal N'/\mathcal N|$, Proposition \ref{prop:fqm2}  implies (in the notation there) that $g(x) =\frake_m -x \in \mathcal N'/\mathcal N$ for $m=0$ or $1$, and
$$
Q_{\mathcal N}(\mathfrak e_m-x)\equiv \frac{m^2 }4 -Q(x)\pmod{1}.
$$
Since $\calN$ has level $\ell$, we have  $\ell Q_{\mathcal N}(\mathfrak e_m-x)\in \Z$, and therefore  $\ell  (\frac{m^2 }4 -Q(x))\in \Z$.
Since $\ell (\frac{1}4 - Q(x))\in \Z$, we find that $m=1$ and $\mathfrak e_1-x\in\mathcal N'/\mathcal N$ as claimed.
\end{proof}

\section{Examples and Remarks} \label{sect:example}

We remark that the intersection  formula in Theorem \ref{theo: mainIntSect}  is effectively computable.  We  give some examples of our computational results in  Table \ref{Table}  in the special case where $d_1=d_2=d$ is an odd fundamental discriminant (equal to the discriminant of $\kay$).

In the table, the first three columns list
the discriminant $d$, the class number $h_d$ of $\kay$, and an ideal $\mathfrak a = [\omega_2, a]= \Z \omega_2 + \Z a$ such that $J(C)\cong \C/\calO_d\times \C/\mathfrak{a}$ where $\omega_2$ is given as a linear combination of  $\omega =\frac{1+\sqrt{d}}{2}$ and $1$, and $a=\norm(\fraka)$. The fourth column describes the polarization  matrix
$$
\alpha =[\alpha_1, \alpha_2,  \beta]= \kzxz {\alpha_1 } { \frac{\bar\beta}a} {\beta} {\alpha_2},   \quad  \alpha_1, \alpha_2 \in \Z_{>0}, \;  \beta \in \mathfrak a, \;   \det \alpha = 1,
$$
where we list all non-product principal polarizations up to equivalence.

The last three columns contain the output of our computation: The quantity $I>0$ such that   $\log I^2$ is the normalized arithmetic intersection number of  $\underline{\mathbf J(C)}$  and the decomposable locus  $\mathcal A_1 \times \mathcal A_1$, the automorphism group $\Aut(C)$, and  the equation of a   curve $C$ with  $\mathbf J(C)\cong( A(\mathfrak a,1), \lambda_0 \circ \alpha)$ (when the equation is relatively simple).  Finally,  $D_n$ denotes the dihedral group of order $n$ and $C_n$ the cyclic group of order $n$.

We have implemented a program\footnote{
Code available on \href{https://github.com/yupeng-ruc/Bad-Reduction-of-Genus-Two-Curves-with-CM}{https://github.com/yupeng-ruc/Bad-Reduction-of-Genus-Two-Curves-with-CM} and in the Zenodo archive under \href{https://doi.org/10.5281/zenodo.20475038}{DOI 10.5281/zenodo.20475038} .}
to compute these output data from the input data $d$, $\mathfrak{a}$, and $\alpha$. We emphasize that the program always gives an equation of the associated genus two curve $C$.


Here are some remarks and observations regarding our computational results.
\begin{enumerate}

\item  The quantity $I^{12}$ matches perfectly with  the column `Explicit computation' in \cite[Tables 1--3]{GRV}.  By the remark before Theorem \ref{theo: mainIntSect}, $\log I$ is the intersection number defined in \cite[Section 5]{GRV} when $A =E \times E$.

\item     Because of the isomorphism in  \eqref{eq:GalAction}, the computational results only depend on the class of $[\mathfrak a]$ in the genus group $\Cl(\kay)/\Cl(\kay)^2$  (with all the different principal polarizations).
In particular, when $-d$ is a prime number, it suffices to consider the case $A= E \times E$ as in \cite{GRV}. If $-d$ is not a prime, we need to consider non-principal genera as well as the principal genus.

\item    When the automorphism group of a curve $C$  is the Dihedral group $D_{12}$ of order $12$, the curve seems to be much simpler than the curves with smaller automorphism groups.

\item  Another interesting observation  from the computation is that the defining field of the curve $C$ is a quadratic extension of its field of moduli if $\Aut(C)=C_2$, and is  the same  as its field of moduli if $\Aut(C)=D_4$ or $D_{12}$. This was already observed in the class number one case in  \cite{GHR}.

\end{enumerate}

\newgeometry{top=2.55cm, left=1.5cm, right=1.5cm} 
{\tiny
\begin{table}
\caption{Computational results}
\label{Table}
\begin{longtable}{>{\centering\arraybackslash}m{.6cm}|>{\centering\arraybackslash}m{.4cm}|>{\centering\arraybackslash}m{1.1cm}|>{\centering\arraybackslash}m{1.8cm}|>{\centering\arraybackslash}m{1.6cm}|>{\centering\arraybackslash}m{1.3cm}|>{\centering\arraybackslash}m{8.5cm}}

\hline
\vspace*{.3em}\small$\boldsymbol{d}$ & \vspace*{.3em}\small$\boldsymbol{h_d}$ & \vspace*{.3em}\small$\boldsymbol{\fraka}$& \vspace*{.3em}\small$\boldsymbol{\alpha}$ & \vspace*{.3em}\small$\boldsymbol{I}$ & \vspace*{.3em}\small\textbf{Aut}($\boldsymbol{C}$) &\vspace*{.3em}\small$\boldsymbol{C}$ \\
\hline
\endfirsthead
\hline
\vspace*{.3em}\small$\boldsymbol{d}$ & \vspace*{.3em}\small$\boldsymbol{h_d}$ & \vspace*{.3em}\small$\boldsymbol{\fraka}$& \vspace*{.3em}\small$\boldsymbol{\alpha}$ & \vspace*{.3em}\small$\boldsymbol{I}$ & \vspace*{.3em}\small\textbf{Aut}($\boldsymbol{C}$) &\vspace*{.3em}\small$\boldsymbol{C}$ \\
\hline
\endhead
\hline
\endfoot
\hline
\endlastfoot

$-11$ &$1$ & $[\omega,1]$ & $[2,2,\omega]$ & $2$ & $D_{12}$ & $y^2 = x^6 + x^3 - 4/11  $\\
\hline
$-15$&$2$&$[\omega,2]$&$[2,2,\omega-2]$& $3$& $D_{12}$ &  $y^2 = x^6 + x^3 + 1/20$\\
\hline
$-19$&$1$ &$[\omega,1]$ & $[2,3,\omega]$ & $2\cdot3$ & $D_4$ & $y^2=-8x^6 + 2040x^5 + 2244x^4 + 5840x^3 +4230x^2 + 4014x + 837$ \\
\hline
$-23$&$3$
&$[\omega,1]$&$[3,3,\omega+1]$ & $5$ & $D_{12}$ & $y^2=x^6+x^3+1/278300(-63504v^2 - 32508v + 60395)$, where $v$ is root of $X^3 + X^2 - 1$\\
\hline
$-31$&$3$
&$[\omega,1]$&$[3,3,\omega]$ & $3^3$ &  $D_4$ &\\
\hline
\multirow{4}{*}{$-35$}&\multirow{4}{*}{$2$}
&$[\omega,1]$&$[2,5,\omega]$ & $2^2\cdot5$ & $D_4$  \\
\cline{3-7}
&&$[\omega,1]$& $[3,4,\omega+1]$ & $2^2\cdot5$ & $D_4$  \\
\cline{3-7}
&&$[\omega,3]$ & $[2,2,\omega]$ & $2^3$ & $D_{12}$ & $y^2=x^6+x^3+1/2645(722-54\sqrt{-7})$\\
\cline{3-7}
&&$[\omega,3]$ & $[2,3,\omega-3]$ & $2^2\cdot 7$ & $D_4$ \\
\hline
\multirow{2}{*}{$-39$}&\multirow{2}{*}{$4$}
&$[\omega,1]$&$[5,8,2\omega-1]$ & $3^3$ & $D_{12}$ & $y^2 = x^6 + x^3 +1/324(19-14\sqrt{13})$\\
\cline{3-7}
&&$[\omega,2]$& $[2,3,\omega]$ & $3^2\cdot7$ & $D_4$  \\
\hline
\multirow{2}{*}{$-43$}&\multirow{2}{*}{$1$}
&$[\omega,1]$& $[2,6,\omega]$ & $2^2\cdot3\cdot5$ & $D_4$ & $y^2 = -8x^6 + 97512x^5 + 121188x^4 +584432x^3 + 410022x^2 + 805482x + 184653$\\
\cline{3-7}
&&$[\omega,1]$& $[3,4,\omega]$ & $2\cdot3\cdot5\cdot 7$ & $C_2$  \\
\hline
\multirow{4}{*}{$-47$}&\multirow{4}{*}{$5$}
&$[\omega,1]$& $[3,5,\omega+1]$ & $5^3$ & $D_4$ \\
\cline{3-7}
&&$[\omega,1]$& $[7,17,3\omega+2]$ & $5\cdot11$ & $D_{12}$ & $y^2=x^6+x^3+1/568700(-259200v^4 + 589896v^3 - 739800v^2 + 658476v - 295873)$, where $v$ is a root of $X^5 - 2X^4 + 2X^3 - X^2 + 1$ \\
\hline
\multirow{5}{*}{$-51$}&\multirow{5}{*}{$2$}
&$[\omega,1]$& $[2,7,\omega]$ & $2^3\cdot3^2$ & $D_4$ \\
\cline{3-7}
&&$[\omega,1]$& $[4,4,\omega+1]$ & $2^3\cdot3^2$ & $D_4$\\
\cline{3-7}
&&$[\omega+1,3]$ & $[2,3,\omega+1]$ & $2^2\cdot 3\cdot 7$ & $D_4$ \\
\cline{3-7}
&&$[\omega+1,3]$ & $[2,9,2\omega-1]$ & $2^2\cdot 3$ & $D_{12}$ & $y^2 = x^6 + x^3 + 4/17$ \\
\cline{3-7}
&&$[\omega+1,3]$ & $[3,4,\omega+4]$ & $2^2\cdot 3\cdot 7^2$ & $C_2$ \\
\hline
\multirow{3}{*}{$-55$}&\multirow{3}{*}{$4$}
&$[\omega,1]$& $[3,5,\omega]$ & $3^4\cdot5^2$ & $C_2$ \\
\cline{3-7}
&&$[\omega,1]$& $[5,12,2\omega+1]$ & $3^4\cdot11$ & $D_4$ \\
\cline{3-7}
&&$[\omega,2]$ & $[2,4,\omega]$ & $3^4\cdot 5^2$ & $D_4$ \\
\hline
\multirow{5}{*}{$-59$}&\multirow{5}{*}{$3$}
&$[\omega,1]$& $[2,8,\omega]$ & $2^4\cdot11$ & $D_4$ \\
\cline{3-7}
&&$[\omega,1]$& $[3,6,\omega+1]$ & $2^3\cdot13$ & $D_4$ \\
\cline{3-7}
&&$[\omega,1]$& $[4,4,\omega]$ & $2^3\cdot11$ & $D_4$ \\
\cline{3-7}
&&$[\omega,1]$& $[5,12,2\omega-1]$ & $2^4$ & $D_{12}$ & $y^2=x^6+x^3+1/15770051(-778356v^2 - 1535004v + 2918504)$, where $v$ is a root of $X^3 + X^2 - X - 2$\\
\hline
\multirow{3}{*}{$-67$}&\multirow{3}{*}{$1$}
&$[\omega,1]$& $[2,9,\omega]$ & $2\cdot3\cdot5\cdot11$ &  $D_4$ & $y^2 = -8x^6 + 1570200x^5 + 1961652x^4 +
    14130800x^3 + 10112454x^2 + 30032550x + 7091361$\\
\cline{3-7}
&&$[\omega,1]$& $[3,6,\omega]$ & $2\cdot3\cdot5\cdot7\cdot13$ & $C_2$ \\
\cline{3-7}
&&$[\omega,1]$& $[4,5,\omega+1]$ & $2^2\cdot3\cdot5\cdot7\cdot11$ & $C_2$
 \\
\hline
\multirow{6}{*}{$-71$}&\multirow{6}{*}{$7$}
&$[\omega,1]$& $[3,7,\omega+1]$ & $7^2\cdot13$ & $D_4$ \\
\cline{3-7}
&&$[\omega,1]$& $[5,5,\omega+2]$ & $11\cdot17$ & $D_{12}$ & $y^2=x^6+x^3+1/18287790943004(-5244946376292v^6 - 13610475918144v^5 - 1676396435112v^4 +
    10528462150416v^3 + 353986525620v^2 + 5385722941620v + 7579618976207)$, where $v$ is a root of $X^7 + 3X^6 + 2X^5 - X^4 - 2X^3 - 2X^2 - X - 1$ \\
\cline{3-7}
&&$[\omega,1]$& $[7,23,3\omega-1]$ & $7^3\cdot11$ & $D_4$ \\
\hline
\multirow{3}{*}{$-79$}&\multirow{3}{*}{$5$}
&$[\omega,1]$& $[3,7,\omega]$ & $3^5\cdot7^3$ &  $C_2$\\
\cline{3-7}
&&$[\omega,1]$& $[5,16,2\omega-1]$ & $3^5\cdot17$ & $D_4$ \\
\cline{3-7}
&&$[\omega,1]$& $[8,12,2\omega+3]$ & $3^5\cdot7^2$ & $D_4$ \\
\hline
\multirow{14}{*}{$-163$}&\multirow{14}{*}{$1$} &$[\omega,1]$
& $[2,21,\omega]$ & $2^2\cdot3\cdot5\cdot23\cdot29$ & $D_4$ & $y^2 = -8x^6 + 3270840792x^5 + 4088548308x^4 +
    68687654192x^3 + 50493631302x^2 + 352024255062x + 85910688153$\\
\cline{3-7}
&&$[\omega,1]$& $[3,14,\omega]$ & $2\cdot3^2\cdot5\cdot7\cdot11\cdot17\cdot31$ & $C_2$ \\
\cline{3-7}
&&$[\omega,1]$& $[4,11,\omega+1]$ & $2\cdot3\cdot5\cdot7\cdot11\cdot13\cdot23\cdot29$ & $C_2$ \\
\cline{3-7}
&&$[\omega,1]$& $[5,33,2\omega]$ & $2^2\cdot3\cdot5\cdot7\cdot13\cdot17\cdot19\cdot37$ & $C_2$ \\
\cline{3-7}
&&$[\omega,1]$& $[6,7,\omega]$ & $2\cdot3^2\cdot5\cdot7\cdot11\cdot17\cdot19\cdot23$ & $C_2$ \\
\cline{3-7}
&&$[\omega,1]$& $[6,8,\omega+2]$ & $2^2\cdot3\cdot5\cdot7\cdot11\cdot13\cdot29\cdot31$ & $C_2$ \\
\cline{3-7}
&&$[\omega,1]$& $[7,24,2\omega+1]$ & $2^3\cdot3\cdot5\cdot11\cdot13\cdot17\cdot19\cdot23$ & $C_2$ \\
\hline
\end{longtable}
\end{table}}
\restoregeometry

\end{document}